\documentclass[leqno, a4paper]{amsart}

\usepackage{amssymb, amsmath, amsthm}
\usepackage{mathrsfs, mathtools}
\usepackage[utf8]{inputenc}
\DeclareUnicodeCharacter{00A0}{ }
\usepackage{paralist}
\usepackage{stmaryrd}
\usepackage{verbatim}
\usepackage{tikz}
\usepackage[linesnumbered,ruled]{algorithm2e}
\usepackage[noend]{algpseudocode}
\usepackage{tabularx,ragged2e,booktabs,caption}
\usepackage{graphics} %% add this and next lines if pictures should be in esp format
\usepackage{epsfig} %For pictures: screened artwork should be set up with an 85 or 100 line screen
\usepackage{graphicx}  
\usepackage{epstopdf}
\usepackage[colorlinks=true]{hyperref}
\usepackage[utf8]{inputenc}
\usepackage{cleveref}
\usepackage{bm}
\usepackage{array}
\usepackage{float}
\usepackage{tikz}
\usepackage[margin=1.1in]{geometry}
\usepackage{pdfcomment}
\usepackage{pgfplots}
\usepackage{dsfont}
\pgfplotsset{compat=1.14}

\newtheorem{theorem}{Theorem}

\newtheorem{proposition}{Proposition}
\newcommand{\longweak}{\relbar\joinrel\rightharpoonup}
\newcommand{\eps}{\varepsilon}
\newcommand{\R}{\mathbb{R}}
\DeclareMathOperator*{\esssup}{ess\,sup}
\hypersetup{urlcolor=black, citecolor=red}
\numberwithin{equation}{section}

\title[Well-posedness of the Blackstock equation] {Well-posedness and numerical treatment of the Blackstock equation in nonlinear acoustics}
\author[Marvin Fritz, Vanja Nikoli\' c, and Barbara Wohlmuth]{}
\subjclass{Primary: 35; Secondary: 35L70}
 \keywords{nonlinear wave equation, well-posedness, energy decay, nonlinear acoustics.}

\email{marvin.fritz@ma.tum.de}
\email{vanja.nikolic@ma.tum.de}
\email{wohlmuth@ma.tum.de}
\crefname{section}{§}{§§}
\Crefname{section}{§}{§§}
\crefformat{section}{§#2#1#3}
\thanks{$^*$ Corresponding author: Marvin Fritz, \url{marvin.fritz@ma.tum.de}}
\begin{document} 
\maketitle
\centerline{\scshape Marvin Fritz$^*$, Vanja Nikoli\' c, and Barbara Wohlmuth} 
\medskip
{\footnotesize
   \centerline{Technical University of Munich, Department of Mathematics, Chair of Numerical Mathematics}
   \centerline{Boltzmannstra\ss e 3, 85748 Garching, Germany}
} 
\vspace{12mm}

\begin{abstract}
We study the Blackstock equation which models the propagation of nonlinear sound waves through dissipative fluids.  Global well-posedness of the model with homogeneous Dirichlet boundary conditions is shown for small initial data. To this end, we employ a fixed-point technique coupled with well-posedness results for a linearized model and appropriate energy estimates. Furthermore, we obtain exponential decay for the energy of the solution. We present additionally a finite element-based method for solving the Blackstock equation and illustrate the behavior of solutions through several numerical experiments. 
\end{abstract}
\vspace{14mm}
%%%%%%%%%%%%%%%%%%%%%%%%%%%%%%%%%%% Introduction %%%%%%%%%%%%%%%%%%%%%%%%%%%%%%%%%%%%
\section{Introduction}
The goal of the present work is to provide well-posedness and a numerical study of an initial-boundary value problem for the Blackstock equation
\begin{align} \label{Blackstock_introduction}
\psi_{tt}-(c^2+B/A \, \psi_t)\Delta \psi-b \Delta \psi_t=2 \nabla \psi \cdot \nabla \psi_t,
\end{align}
which serves as a model for nonlinear ultrasound propagation in thermoviscous fluids. The equation is given in terms of the acoustic velocity potential $\psi$, with $c$ denoting the speed of sound, $b$ the sound diffusivity, and $B/A$ the parameter of nonlinearity of the medium.  \\ 
\indent Our research is motivated by many applications of high-intensity focused ultrasound (HIFU) in medicine and industry. For instance, HIFU is widely used in the noninvasive treatment of kidney stones \cite{chaussy1980extracorporeally, ikeda2006cloud, yoshizawa2009high}. In recent years, there have been many studies on the benefits of HIFU in the treatment of cancer in a number of organs, including the liver \cite{yang1991high}, prostate \cite{blana2004high, poissonnier2007control}, and brain \cite{coluccia2014first, maloney2015emerging}.  \\
\indent The equation \eqref{Blackstock_introduction} was derived by Blackstock in \cite{Blackstock} as a one-equation approximation of the compressible Navier-Stokes system. Since it later appeared independently in the works of Crighton \cite{crighton1979model} as well as Lesser and Sebass \cite{lesser1968structure}, it is also referred to as the Blackstock-Lesser-Seebass-Crighton equation; see \cite{jordan2012propagation}. In a study performed in $1$D \cite{ChristovJordan}, this equation was shown to be the most consistent one among the weakly nonlinear acoustic models in the lossless case $b\to 0$. Furthermore, it was shown in \cite{christov2016acoustic} that the Blackstock equation agrees to a large extent with the exact result based on the fully nonlinear theory in the small Mach number limit. Nevertheless, the Blackstock equation has received less attention in the mathematical literature compared to the well-studied Kuznetsov and the Westervelt equation. \\
\indent In \cite{MizohataUkai}, Mizohata and Ukai investigated the Kuznetsov equation in the potential formulation and showed global well-posedness with Dirichlet boundary conditions for small initial data. Kaltenbacher and Lasiecka later considered the pressure-velocity formulation of the Kuznetsov equation \cite{KaltenbacherLasiecka_Kuznetsov} and the pressure formulation of the  Westervelt equation \cite{KaltenbacherLasiecka_Westervelt, KaltenbacherLasiecka_both} with different boundary conditions and showed well-posedness together with the energy decay. Meyer and Wilke  generalized and improved their results \cite{meyer2011optimal, meyer2012global} by employing a maximal $L^p$-regularity approach. Recently in \cite{Tani}, Tani studied the Cauchy  problem in $\mathbb{R}^3$ for a mathematically more general model than the Blackstock equation and showed global-in-time existence. The model studied in \cite{Tani} is also referred to in the literature as the 
Rasmussen-S\o rensen-Gaididei-Christiansen equation; see \cite{christov2016acoustic, rassmusen2011interacting}. The Blackstock equation can fit into the framework of a general evolution model studied in \cite{ebihara1979some}, where existence and uniqueness results are provided for very regular initial data.  \\
\indent In the present work, we study the Blackstock equation with homogeneous Dirichlet boundary conditions and initial data in $H_0^1 \cap H^3 \times H_0^1 \cap H^2$ and provide results on local well-posedness, global well-posedness, and exponential decay rates for the energy of solutions. To obtain local well-posedness, we employ a fixed point approach, relying on the well-posedness results for a linearized equation. Global well-posedness follows from appropriate energy estimates. Furthermore, we present a finite element-based numerical treatment of the model with different boundary conditions. \\ 
\indent The rest of the paper is organized as follows. In Section \ref{section:problem_setting}, we lay out the problem, set up the notation, and present important theoretical results for future use. Section \ref{section:linearized_Blackstock} is devoted to the well-posedness results for a linearization of the Blackstock equation. In Section \ref{section:local_wellposedness}, we tackle the local well-posedness of the initial-boundary value problem for the Blackstock equation. Section \ref{section:global_wellposedness} is concerned with the global well-posedness and energy decay. In Section \ref{section:numerical_treatment}, we present the numerical solver for the Blackstock equation. Finally, Section \ref{section:numerical_experiments} contains several numerical experiments that illustrate the behavior of the model. 
%%%%%%%%%%%%%%%%%%%%%%%%%%%%%%%%%%% Problem Setting %%%%%%%%%%%%%%%%%%%%%%%%%%%%%%%%%%%%%
\section{Problem setting} \label{section:problem_setting}
%%%%%%%%%%%%%%%%%%%%%%%%%%%%%%%%%%% Models %%%%%%%%%%%%%%%%%%%%%%%%%%%%%%%%%%%%%
There are many nonlinear acoustic models in the literature that serve as approximations of the compressible Navier-Stokes system; we refer the interested reader to the survey \cite{jordan2016survey}. One of the most popular models is the Kuznetsov equation
\begin{align} \label{Kuznetsov_eq}
(1-B/A c^{-2}\psi_t)\psi_{tt}-c^2\Delta \psi-b \Delta \psi_t=2 \nabla \psi \cdot \nabla \psi_t.
\end{align}
The acoustic velocity potential $\psi$ and the acoustic particle velocity $v$ are related by $v=-\nabla \psi$. By employing the approximation $\nabla \psi \cdot \nabla \psi_t \approx c^{-2} \psi_t \psi_{tt}$, we can transform \eqref{Kuznetsov_eq} into the Westervelt equation
\begin{align} \label{Westervelt_eq}
(1-(B/A+2)c^{-2}\psi_t)\psi_{tt}-c^2\Delta \psi-b \Delta \psi_t=0.
\end{align}
We can similarly obtain the Kuznetsov equation from the Blackstock equation \eqref{Blackstock_introduction} by making use of the approximation $\psi_t \Delta \psi \approx c^{-2} \psi_t \psi_{tt}$. The acoustic pressure $u$ and the acoustic velocity potential are related by $u \approx \varrho \psi_t$, where $\varrho$ is the mass density of the medium. \\
\indent In \cite{ChristovJordan, ChristovJordan_Corrigendum}, the three models were compared in $1$D in the inviscid case where $b=0$. The equations were reformulated as first-order systems and compared to the Euler equations, the system corresponding to the Blackstock equation performed the best; see \cite[Figures 1-5]{ChristovJordan}. A difference in the pressure profiles obtained by using the three different models can be observed in our numerical experiments in Section~\ref{section:numerical_experiments}; see Figure~\ref{fig:channel_comparison_Blackstock_Kuznetsov_Westervelt}.
\subsection{The initial-boundary value problem for the Blackstock equation} Let $\Omega$ be a bounded, $C^{1,1}$ regular domain in $\mathbb{R}^d$, where $d \in \{1,2,3\}$, and let $T>0$. We study the following initial-boundary value problem for the Blackstock equation 
\begin{align} \label{Blackstock}
\begin{cases}
\psi_{tt}-c^2(1+k \psi_t)\Delta \psi-b \Delta \psi_t=2 \nabla \psi \cdot \nabla \psi_t \quad \text{in } \Omega \times (0,T),
\smallskip\\ \psi=0 \quad \text{ on }  \partial \Omega \times (0,T),  
\smallskip\\ (\psi, \psi_t)=(\psi_0, \psi_1) \quad \text{ on }   \Omega \times \{t=0\},
\end{cases}
\end{align}
where $k=c^{-2}B/A$. Since the sign of the constant $k$ does not play a significant role when proving well-posedness, we carry out the analysis of the model with the assumption that $k \in \mathbb{R}$.\\
\indent An important task in the mathematical analysis of the Kuznetsov and the Westervelt equations is to avoid degeneracy \cite{KaltenbacherLasiecka_Westervelt, KaltenbacherLasiecka_both, KaltenbacherLasiecka_Kuznetsov}, which can occur if the factor next to $\psi_{tt}$ vanishes. \\
\indent In the well-posedness results for the Blackstock equation, we do not require $1+k \psi_t$ to be positive almost everywhere, only almost everywhere bounded. In particular, the Blackstock equation is allowed to degenerate to
\begin{align} \label{degenerate_eq}
\psi_{tt}-b\Delta \psi_t = 2 \nabla \psi \cdot \nabla \psi_t,
\end{align}
which can be rewritten as $\partial_t(\psi_t-b\Delta \psi-|\nabla \psi|^2)=0$. We show that $\psi_t$ is almost everywhere bounded by first deriving a bound in $L^\infty(0,T; H^2)$ and then by employing the embedding $H^2 \hookrightarrow L^\infty$.
% If $\psi_t=-1/k$, $k\neq 0$, the Blackstock equation degenerates to the equation
%\begin{align} \label{degenerate_eq}
%\psi_{tt}-b\Delta \psi_t = 2 \nabla \psi \cdot \nabla \psi_t,
%\end{align}
%which can be rewritten as $\partial_t(\psi_t-b\Delta \psi-|\nabla \psi|^2)=0$. 
% During the course of our proofs we derive a bound on 
%$\|\psi_t\|_{L^\infty(\Omega \times (0,T))}$ to ensure that $c^2(1+k\psi_t)$ is almost everywhere bounded. Such estimate arises from bounding $\|\psi_t\|_{L^\infty(\Omega \times (0,T))}$ by the initial conditions in the appropriately chosen norm.
%%%%%%%%%%%%%%%%%%%%%%%%%%%%%%%%%%% Notation and Preliminaries %%%%%%%%%%%%%%%%%%%%%%%%%%%%%%%%%%%%%
\subsection{Notation} \label{section:notation} Before proceeding further, let us briefly set the notation. We often omit the domain when denoting a Banach space and only write $L^p$, $H^m$, $W^{m,p}$. However, in case of $d$-dimensional vector functions in one of these spaces, we write $L^p(\Omega,\mathbb{R}^d)$, $H^m(\Omega,\mathbb{R}^d)$, $W^{m,p}(\Omega,\mathbb{R}^d)$. By a slight abuse of notation, we don't make this distinction when writing the norm and always use $|\cdot|_{L^p}$, $|\cdot|_{H^m}$, $|\cdot|_{W^{m,p}}$ . \\
\indent For a given Banach space $X$, we equip the Bochner space
$$L^p(0,T;X) =\{\, u:(0,T) \to X ~:~ u \text{ Bochner measurable, } \int_0^T |u(t)|_X^p \, \textup{d}t < \infty \, \},$$
where $1 \leq p < \infty$, with the norm $\|u\|_{L^p X}^p = \int_0^T |u(t)|_X^p \, \textup{d}t$. For $p=\infty$, we modify it standardly and define the norm in $L^\infty(0,T;X)$ as $\|u\|_{L^\infty X}=\esssup_{t \in (0,T)} |u(t)|_X.$ \\
\indent Furthermore, for given Banach spaces $X$ and $Y$, we equip the Sobolev-Bochner space
$$W^{1,p,q}(0,T;X,Y)=\{\, u \in L^p(0,T;X): u_t \in L^q(0,T;Y) \, \}, \quad 1 \leq p, q \leq \infty,$$
with the norm $\|u\|_{W^{1,p,q}XY}=\|u\|_{L^p X} + \|u_t\|_{L^q X}$. For short-hand notation when $p=q$ and $X=Y$ we use
$$W^{1,p}(0,T;X)=W^{1,p,p}(0,T;X,X).$$  
\indent Throughout the paper $C$ stands for a generic positive constant and $x \lesssim y$ stands for $x \leq C y$. 
\subsection{Theoretical preliminaries} We collect here some theoretical results which we often use in our proofs. We frequently employ Young's $\varepsilon$-inequality \cite[Appendix B]{Evans}
\begin{align} \label{Young_inequality}
    xy \leq \varepsilon x^p+C(\varepsilon)y^q, \quad \text{where} \ x,y>0, \, \varepsilon>0, \, 1<p,q<\infty, \ \frac{1}{p}+\frac{1}{q}=1,
\end{align}
and $C(\varepsilon)=(\varepsilon p)^{-q/p}q^{-1}$, as well as the following two special cases of the Gagliardo-Nirenberg inequality \cite[Theorem 1.24]{Roubicek}
\begin{equation}
\begin{aligned} \label{GN_inequality}
 |f|_{L^4} &\leq K_{\text{GN}}\, |f|_{L^2}^{1-d/4} |\nabla f|_{L^2}^{d/4}, &&\quad \text{for} \ f \in H^1_0, \ K_\text{GN}<\infty, \\
 |\nabla f|_{L^4} &\lesssim  K_{\text{GN}}\, |\nabla f|_{L^2}^{1-d/4} |\Delta f|_{L^2}^{d/4}, &&\quad \text{for} \ f \in H^1_0 \cap H^2. 
\end{aligned}
\end{equation}
Let $u$ and $v$ be non-negative continuous functions and $C_1,C_2<\infty$ non-negative constants such that
$$u(t) + v(t) \leq C_1+C_2 \int_0^t u(s) \, \textup{d}s \quad \text{ for all } t\in [0,T].$$ Then the following modification of Gronwall's inequality holds \cite[Lemma 3.1]{garcke2017well}
\begin{equation} u(t) + v(t) \leq C_1 e^{C_2 T} \quad \text{ for all } t\in[0,T].
\label{Gronwall_inequality}
\end{equation}
\indent In our well-posedness proofs, we rely heavily on different Sobolev embeddings. For future use we introduce here the embedding constants $K_1, \ldots, K_5<\infty$ by
\begin{equation} \label{embed_constants}
\begin{aligned}
|f|_{L^4} &\leq K_1 |\nabla f|_{L^2}, && \quad f \in H_0^1 \hookrightarrow L^4, \\
|\nabla f|_{L^4} &\leq K_2 |\Delta f|_{L^2}, &&\quad f\in H_0^1\cap H^2 \hookrightarrow W_0^{1,4}, \\
|f |_{L^\infty} &\leq K_3 |\nabla f|_{L^4}, &&\quad f \in W_0^{1,4} \hookrightarrow L^\infty, \\
|\nabla f|_{L^\infty} &\leq K_4 |\Delta f|_{L^4}, &&\quad f \in H_0^1\cap W^{2,4} \hookrightarrow W_0^{1,\infty}, \\[-.1cm]
|\Delta f|_{L^4}&\leq K_5 \big( |\Delta f|_{L^2}^2+|\nabla \Delta f|_{L^2}^2\big)^{1/2},  &&\quad f \in H_0^1 \cap H^3 \hookrightarrow H_0^1 \cap W^{2,4}, %\\
%|f|_{L^2} &\leq K_{\text{P}} |\nabla f|_{L^2}, &&\quad f \in H_0^1,
\end{aligned}
\end{equation}
recalling that we assumed that the dimension $d\leq 3$. We also introduce the Poincar\'e constant $K_{\text{P}}<\infty$, where
\begin{equation} \label{Poincare_constant}
\begin{aligned}
|f|_{L^2} &\leq K_{\text{P}} |\nabla f|_{L^2}, &&\quad f \in H_0^1.
\end{aligned}
\end{equation}
Furthermore, we often employ the compact embedding \cite[Corollary 4]{simon1986compact}
\begin{align*}
W^{1,\infty,2}(0,T;X,Z) \hookrightarrow \hookrightarrow C([0,T];Y),   
\end{align*}
for a given Gelfand triple $X \hookrightarrow \hookrightarrow Y \hookrightarrow Z$, as well as the the continuous embedding \cite[Ch.1, Theorem 3.1]{lions1972non}
\begin{align*}
W^{1,2,2}(0,T;X,Y) \hookrightarrow C([0,T];[X,Y]_{1/2}),
\end{align*}
for $X \hookrightarrow Y$. Here $[X,Y]_{1/2}$ denotes the interpolation space between $X$ and $Y$; see \cite[Ch.1, Definition 2.1]{lions1972non} for a  precise definition. 
%%%%%%%%%%%%%%%%%%%%%%%%%%%%%%%%%%% Linearized Blackstock %%%%%%%%%%%%%%%%%%%%%%%%%%%%%%%%%%%%
\section{Results for the linearized Blackstock equation} \label{section:linearized_Blackstock}
We begin by considering an initial-boundary value problem for a linearization of the Blackstock equation \eqref{Blackstock} which has space-time variable coefficients
\begin{align} \label{IBVP_linearizedBlackstock}
\begin{cases}
\psi_{tt}-\alpha(x,t)\Delta \psi-b \Delta \psi_t=\beta(x,t) \cdot \nabla \psi_t+f(x,t) \quad \text{in } \Omega \times (0,T),
\smallskip \\ \psi=0 \quad \text{ on }  \partial \Omega \times (0,T),  
\smallskip \\ (\psi, \psi_t)=(\psi_0, \psi_1) \quad \text{ on }   \Omega \times \{t=0\}. 
\end{cases}
\end{align}
Here $\alpha$ and $f$ are scalar functions and $\beta$ is a vector-valued function. We will specify their regularity in the upcoming propositions. A result on well-posedness of \eqref{IBVP_linearizedBlackstock} when $\beta=0$ can be found in \cite[Section 7.2]{kaltenbacher2014efficient} with stronger assumptions on $\alpha$ than Proposition~\ref{Proposition1} requires: $\alpha \in W^{1,1}(0,T; L^\infty(\Omega))$, $\alpha$ almost everywhere bounded from below by a positive constant, and $\|\alpha_t\|_{L^2 L^\infty}$ sufficiently small. \\
\indent We refer to the partial differential equation in \eqref{IBVP_linearizedBlackstock} as the linearized Blackstock equation. After proving well-posedness for this linear model, we insert $\alpha=c^2(1+kv_t)$, $\beta=2 \nabla v$, $f=0$ and define the operator $\mathcal{F}: v \mapsto \psi$ on which we employ a fixed-point theorem. Note that the model \eqref{IBVP_linearizedBlackstock} is more general than we immediately need for the fixed-point technique. However, the function $f$ comes into play when we set out to prove the contraction property of $\mathcal{F}$. \\
\indent We prove two well-posedness results for the linearized Blackstock equation. For the well-posedness of the nonlinear model in a general setting, Proposition \ref{Proposition2} is the relevant one. In the one-dimensional case the conclusion of Proposition \ref{Proposition1} is sufficient to show well-posedness. 
%%%%%%%%%%%%%%%%%%%%%%%%%%%%%%%%%%%%%% Proposition 1 %%%%%%%%%%%%%%%%%%%%%%%%%%%%%%%%
\subsection{Well-posedness of the linear model} We first show existence of a unique solution of the initial-boundary value problem \eqref{IBVP_linearizedBlackstock} when the initial data belongs to $(H_0^1 \cap H^2,H_0^1)$.
\begin{proposition} \label{Proposition1}
	Let $\Omega \subset \mathbb{R}^d$, where $d \in \{1,2,3\}$, be a bounded, open, and $C^{1,1}$ regular domain. Let $b>0$ and let the following regularity assumptions hold
	\begin{itemize}
	    \item $\alpha \in L^\infty(\Omega \times (0,T))$, \smallskip
	  %  \item there exist $\underline{\alpha}_0, \overline{\alpha}$ such that $0 \leq \underline{\alpha}_0 \leq \alpha (x,t) \leq \overline{\alpha}_0$ a.e. in $\Omega \times (0,T)$, \smallskip 
	    \item $\beta \in L^\infty(0,T; L^4(\Omega,\R^d))$, \smallskip
	    \item $f \in L^2(0,T;L^2)$, \smallskip 
	    \item $(\psi_0,\psi_1) \in (H_0^1 \cap H^2,H_0^1).$
	\end{itemize}
    Then for every $T>0$, the initial-boundary value problem \eqref{IBVP_linearizedBlackstock} for the linearized Blackstock equation admits a unique solution $\psi$ in the $L^2(0,T;L^2)$ sense that satisfies
	\begin{align*}
	\begin{cases}
	\psi \in C([0,T];H_0^1\cap H^2), \smallskip \\
	\psi_t \in C([0,T];H_0^1) \cap  L^2(0,T;H^2), \smallskip \\
	\psi_{tt} \in L^2(0,T;L^2).
	\end{cases}
	\end{align*}
Furthermore, the following energy estimate holds
\begin{equation} \label{Prop1Energy}
	\begin{aligned}
	& \| \psi \|^2_{L^\infty H^2} +\| \psi_t \|^2_{L^\infty H^1}+\| \psi_t\|^2_{L^2H^2}  +\|\psi_{tt}\|_{L^2L^2}^2\\
	\leq& \, C_{\textup{P}_1} (T,\alpha,\beta)\big( |\psi_0|_{H^2}^2+|\psi_1|_{H^1}^2 +\|f\|_{L^2L^2}^2\big).
	\end{aligned}
	\end{equation}
The constant above is given by
\begin{equation*} %\label{CP1}
\begin{aligned}
C_{\text{P}_1}(T,\alpha,\beta) = C \exp\Big( C T \big(1+\|\alpha\|_{L^\infty(\Omega \times (0,T))}+\|\alpha\|_{L^\infty(\Omega \times (0,T))}^2 + \|\beta\|_{L^\infty L^4}^2 + \|\beta\|_{L^\infty L^4}^\frac{2}{1-d/4}\big) \Big) .
 \end{aligned}
 \end{equation*}
\end{proposition}
\begin{proof}
\noindent We employ the Faedo-Galerkin approach \cite{Evans, Roubicek, Temam} to show well-posedness, where we approximate our problem in space and then show that a sequence of solutions of the approximate problems converges to a solution of the original problem \eqref{IBVP_linearizedBlackstock}. \\
\indent Note that coefficient $\alpha$ does not have to be positive in \eqref{IBVP_linearizedBlackstock}. We also remark that we assumed that $\alpha \in L^\infty(\Omega \times (0,T))$, since it is more general than assuming $\alpha$ to be in $L^\infty(0,T;L^\infty)$. An example of a function which is in $L^\infty(\Omega \times (0,T))$, but not in $L^\infty(0,T;L^\infty)$ can be found in \cite[Example 1.42]{Roubicek}. \medskip \\
\noindent \textbf{Discretization in space.}
To discretize our problem in space, we select smooth functions $\{w_k\}_{k \in \mathbb{N}}$ which form an orthogonal basis of $H_0^1(\Omega)\cap H^2(\Omega)$. For our analysis it is convenient to choose the eigenfunctions of the Laplace-Dirichlet problem 
$$\begin{cases} -\Delta w_k = \lambda_k w_k, \\
w_k|_{\partial \Omega}=0, 
\end{cases}$$
where $k \in \mathbb{N}$. Then the basis is additionally orthonormal in $L^2(\Omega)$; see \cite[Ch.1, Section 1.7]{lions1969quelques}. For a fixed $n \in \mathbb{N}$, we denote by 
\begin{align*}
V_n=\text{span}\,  \{w_1, \ldots, w_n\}
\end{align*}
the finite-dimensional subspace of $H_0^1(\Omega) \cap H^2(\Omega)$ spanned by the first $n$ vectors of the basis. We then consider Galerkin approximations
\begin{align*}
& \psi^n(x,t)=\displaystyle \sum_{i=1}^n \xi_i(t) w_i(x), 
\end{align*}
where $\xi_i: (0,T) \to \mathbb{R}$ are coefficient functions for $i\in[1,n]$. We approximate the initial data by
\begin{align*}
& \psi^n_0(x)=\displaystyle \sum_{i=1}^n \xi_{i, 0}\, w_i(x),  \\
& \psi^n_1(x)=\displaystyle \sum_{i=1}^n \xi_{i,1} \, w_i(x).
\end{align*}
Here coefficients $\xi_{i,0}, \xi_{i,1} \in \mathbb{R}$ are conveniently chosen as
\begin{equation*}
\begin{aligned}
    \xi_{i,0} &=(\psi_0,w_i)_{L^2}, \\
    \xi_{i,1} &=(\psi_1,w_i)_{L^2},
    \end{aligned}
\end{equation*}
for $i \in [1,n]$. In this way it follows by construction that 
\begin{equation}  \label{Prop1GalerkinIC}
\begin{aligned}
  \|\psi_0^n\|_{H^2} &\leq \|\psi_0\|_{H^2} &&\text{and} &&\psi_0^n \longrightarrow \psi_0  \text{ in } H_0^1\cap H^2, \\
  \|\psi_1^n\|_{H^1} &\leq \|\psi_1\|_{H^1} &&\text{and} &&\psi_1^n \longrightarrow \psi_1 \text{ in } H_0^1;
\end{aligned}
\end{equation}
see \cite[Lemma 7.5]{robinson2001infinite}. We will use these bounds later to derive energy estimates that are uniform with respect to $n$. Now we can consider the following approximation of our original problem
\begin{equation} \label{Prop1Galerkin}
\begin{aligned}
\begin{cases}
( \psi^n_{tt} , v )_{L^2} - (\alpha \Delta \psi^n,v)_{L^2}-b(\Delta \psi_t^n,v)_{L^2}= (\beta \cdot \nabla \psi_t^n,v )_{L^2}+(f,v)_{L^2}, \smallskip \\
\text{for every $v \in V_n$ pointwise a.e. in $(0,T)$}, \smallskip \\
\psi^n(0)=\psi_0^n, \ \psi^n_t(0)=\psi_1^n.
\end{cases}
\end{aligned}
\end{equation}
We can rewrite this approximated problem as a system of $n$ ordinary differential equations for the coefficient functions $\xi_i$, where $i\in[1, n]$. First we note that \eqref{Prop1Galerkin} is equivalent to 
\begin{equation} \label{Prop1GalerkinTime}
\begin{cases}
\displaystyle \sum_{i=1}^n \int_{\Omega}\Bigl ((\xi_i)_{tt} \, w_i w_j -\xi_i \, \alpha \Delta w_i w_j- (\xi_i)_t \, b \Delta w_i  w_j -(\xi_i)_t \, \beta \cdot \nabla w_i w_j  \Bigr)\, \textup{d}x=\int_{\Omega} f w_j \, \textup{d}x, \smallskip  \\
\xi_i(0)=\xi_{i,0}, \quad (\xi_{i})_t(0)=\xi_{i,1}, \quad j \in [1, n].
\end{cases}
\end{equation}
We then introduce matrices $M^n=[M_{ij}]$, $K^n(t)=[K_{ij}]$, $C^n(t)=[C_{ij}]$, and vector $F^n(t)=[F_j]$ whose elements are computed according to
\begin{equation*}
\begin{aligned}
& M_{ij}=(w_i, w_j)_{L^2}=\delta_{ij},\\
& K_{ij}(t)=-(\alpha \Delta w_i, w_j)_{L^2}, \\
& C_{ij}(t)=-b \, (\Delta w_i, w_j)_{L^2}-( \beta \cdot \nabla w_i, w_j)_{L^2},\\
& F_j(t)=(f, w_j)_{L^2},
\end{aligned}
\end{equation*}
for $i, j \in [1,n]$ and almost every $t \in (0,T)$. We note that $M^n$ is the identity matrix, while matrices $K^n(t)$ and $C^n(t)$ and vector $F^n(t)$ are well-defined thanks to the fact that
\begin{align*}
& (\alpha(t) \Delta w_i, w_j)_{L^2} \leq |\alpha(t)|_{L^2}  |\Delta w_i|_{L^2} | w_j|_{L^\infty}, 
\\& (f(t), w_j)_{L^2} \leq |f(t)|_{L^2} | w_j|_{L^2}, 
\\& (\beta(t) \cdot \nabla w_i, w_j)_{L^2} \leq |\beta(t)|_{L^2}  |\nabla w_i|_{L^4} | w_j|_{L^4}, 
\end{align*}
almost everywhere in time. By introducing vectors $\xi^n=[\xi_1 \ldots \xi_n]^T$, $\xi_0^n=[\xi_{1,0} \ldots \xi_{n, 0}]^T$, and $\xi_1^n=[\xi_{1,1} \ldots \xi_{n,1}]^T$, problem \eqref{Prop1GalerkinTime} can be rewritten in the form of a matrix equation
\begin{equation} \label{Prop1GalerkinMatrix_old}
\begin{cases}
\xi_{tt}^n+K^n(t)\xi^n+C^n(t) \xi_t^n=F^n(t), \\
\xi^n(0)=\xi_0^n, \quad \xi^n_t(0)=\xi_1^n.
\end{cases}
\end{equation}
With $u^n=[\xi^n \ \ \xi^n_t]^T$, $u_0^n=[\xi_0^n \ \ \xi_1^n]^T$, $A^n(t)=\begin{bmatrix}
0 & I \\ -K^n(t) & -C^n(t) \end{bmatrix}$, and $G^n=[0 \ \ F^n(t)]^T$ we can further rewrite \eqref{Prop1GalerkinMatrix_old} as
\begin{equation} \label{Prop1GalerkinODE}
\begin{cases} 
\dfrac{\textup{d}}{\textup{d}t} u^n =A^n(t) u^n+G^n(t), \\
 u^n(0)=u_0^n.
\end{cases}
\end{equation}
The existence now follows from the standard theory of ordinary differential equations. It can be seen that the right-hand side of the ODE in \eqref{Prop1GalerkinODE} $$g(t,r)=A^n(t)r+G^n(t)$$ is continuous with respect to $r$ for a fixed $t$ and it is measurable with respect to $t$ for a fixed $r$. According to Carathéodory's theorem \cite[Theorem 1.44]{Roubicek}, we have local-in-time existence of an absolutely continuous solution $u^n$ and thus $\psi^n, \psi_t^n$ on some sufficiently short interval $[0, T_n]$.  Furthermore, we can estimate the second time derivative of the solution of \eqref{Prop1GalerkinMatrix_old} in the following way
$$\begin{aligned} |\xi_{tt}^n|_{L^2(0,T_n)}^2 =& \,|-K^n\xi^n-C^n\xi^n_t+F^n|_{L^2(0,T_n;\R^d)}^2 \\
\leq& \, |\xi^n|^2_{L^\infty(0,T_n)} \sum_{i,j}^n \int_0^{T_n}  \big|(\alpha \Delta w_i, w_j)_{L^2} \big|^2 \, \textup{d}t +\|f\|_{L^2L^2}^2\\
&+|\xi_t^n|_{L^\infty(0,T_n)}^2  \sum_{i,j}^n \int_0^{T_n}  \big| (b\Delta w_i-\beta\cdot \nabla w_i, w_j)_{L^2}  \big|^2\, \textup{d}t \\
\lesssim& \,   \|\alpha\|_{L^2L^2}^2+\|f\|_{L^2L^2}^2+\|\beta\|_{L^2L^2}^2<\infty,
\end{aligned}$$
where we note that the vector norm is chosen to be the Euclidean norm. We can conclude that for every fixed $n \in \mathbb{N}$, a solution $\psi^n$ of \eqref{Prop1Galerkin} exists with the regularity $C^1([0,T_n];V_n)\cap H^2(0,T_n;V_n).$ The upcoming energy estimates allow us to extend the solution to the whole interval $[0,T]$. \medskip \\
%%%%%%%%%%%%%%%%%%%%%%%%%% ENERGY ESTIMATES %%%%%%%%%%%%%%%%%%%%%%%%%%%%%%%%%%%%%%%%%%
\noindent \textbf{Energy estimates for approximate solutions.} Next we set out to derive energy estimates for $\psi^n$ which are uniform with respect to $n$. To this end, we test our approximate problem \eqref{Prop1Galerkin} with four different test functions. To be able to later prove the well-posedness of the nonlinear model, it is important that in these estimates we track the initial data and the norms of $\alpha, \beta$, and $f$ precisely. The embedding constants are typically not singled out in the final estimates, but contained in a generic constant $C$. \medskip \\ 
%%%%%%%%%%%%%%%%%%%%%%%%%%%%%%%%%%%%%%%%%%%%%%%%%%%%%%%%%%%%%%%%%%%%%%%%%%%%%%%%%%%%%%%%%%%%
\noindent\textit{Testing with $\psi^n_t$.} Since $\psi^n_t(s) \in V_n$ for all $s \in [0,T_n]$, we are allowed to take $\psi^n_t(s)$ as a test function in \eqref{Prop1Galerkin}. We test \eqref{Prop1Galerkin} with $\psi^n_t(s)$, integrate with respect to time from $0$ to $t$, and perform integration by parts in time on the term containing the second time derivative. Together with employing H\"older's inequality, these actions yield the estimate
\begin{equation*} 
\begin{aligned}
&\frac12 | \psi^n_t(t) |^2_{L^2}+ b \| \nabla \psi_t^n \|_{L^2L^2}^2 
\\\leq& \, \frac12 |\psi_1^n|_{L^2}^2 +  \|\alpha\|_{L^\infty(\Omega \times (0,T))} \|\Delta \psi^n\|_{L^2L^2} \|\psi_t^n\|_{L^2L^2}
\\ & +\|\beta\|_{L^\infty L^4} \|\nabla \psi_t^n\|_{L^2L^2} \|\psi_t^n\|_{L^2 L^4}+ \|f\|_{L^2L^2} \|\psi^n_t\|_{L^2 L^2},
\end{aligned}
\end{equation*}
for all $t \in [0,T_n]$. We can use Poincaré's inequality to estimate $\|\psi_t^n\|_{L^2L^2}$ within the terms on the right-hand side. For the $\beta$-term, we make use of the Gagliardo-Nirenberg's inequality \eqref{GN_inequality} 
\begin{equation*} 
\begin{aligned}
& \|\beta\|_{L^\infty L^4} \|\nabla \psi_t^n\|_{L^2L^2} \|\psi_t^n\|_{L^2 L^4} \leq K_\text{GN} \| \beta\|_{L^\infty L^4} \|\nabla \psi_t^n\|_{L^2L^2}^{1+d/4} \|\psi_t^n\|_{L^2 L^2}^{1-d/4}.
\end{aligned}
\end{equation*}
By additionally employing Young's $\eps$-inequality \eqref{Young_inequality} with either $p=\frac{2}{1+d/4}, q=\frac{2}{1-d/4}$ or $p=q=2$, we arrive at
\begin{equation} \label{Prop1Energy1}
\begin{aligned}
& \frac12 | \psi^n_t(t) |^2_{L^2}+ (b-3\eps) \| \nabla \psi_t^n \|_{L^2L^2}^2 
\\ \leq& \, \frac12 |\psi_1^n|_{L^2}^2 +C\|\alpha\|_{L^\infty(\Omega \times (0,T))}^2  \| \Delta \psi^n \|_{L^2L^2}^{2} + C\|\beta\|_{L^\infty L^4}^\frac{2}{1-d/4} \| \psi^n_t \|^{2}_{L^2L^2}  + C \|f\|_{L^2L^2}^2,
\end{aligned}
\end{equation}
for all $t \in [0,T_n]$ and some conveniently chosen $\eps>0$. We note that we are still missing a bound on $\|\Delta \psi^n\|_{L^2L^2}^2$ in \eqref{Prop1Energy1}, which we get by testing our problem \eqref{Prop1Galerkin} with $-\Delta \psi^n(s)$, $s\in [0,T_n]$.\medskip \\
%%%%%%%%%%%%%%%%%%%%%%%%%%%%%%%%%%%%%%%%%%%%%%%%%%%%%%%%%%%%%%%%%%%%%%%%%%%%%%%%%%%%%%%%%%%%
\noindent \textit{Testing with $-\Delta \psi^n$.} Since $\Delta \psi^n(s) \in V_n$ for all $s \in [0,T_n]$, we indeed are allowed to employ $-\Delta \psi^n(s)$ as a test function in \eqref{Prop1Galerkin}. After integrating the resulting equation with respect to time and applying Young's $\varepsilon$-inequality, we find that
\begin{equation} \label{Prop1Energy2}
    \begin{aligned}
    & \frac{b}{2} |\Delta \psi^n(t)|^2_{L^2} 
    \\\leq & \, \frac{b}{2} |\Delta \psi_0^n|_{L^2}^2+ \eps \|\psi^n_{tt}\|_{L^2 L^2}^2 + \eps \|\Delta \psi_t^n\|_{L^2L^2}^2 + \frac{1}{2} \|f\|_{L^2L^2}^2\\
    & +C (\|\alpha\|_{L^\infty L^\infty}+\|\beta\|_{L^\infty L^4}^2+1) \|\Delta \psi^n \|_{L^2L^2}^2,
    \end{aligned} 
\end{equation}
for all $t\in [0,T_n]$ and some $\varepsilon>0$. We observe that we still need a bound on $\|\psi_{tt}^n\|_{L^2L^2}^2$ and on $\|\Delta \psi_t^n\|_{L^2L^2}^2$. Thus, we continue by testing the problem \eqref{Prop1Galerkin} first with $\psi_{tt}^n(s)$, and then with $-\Delta \psi_t^n(s), s \in [0,T_n]$.\medskip \\
%%%%%%%%%%%%%%%%%%%%%%%%%%%%%%%%%%%%%%%%%%%%%%%%%%%%%%%%%%%%%%%%%%%%%%%%%%%%%%%%%%%%%%%%%%%%
\noindent \textit{Testing with $\psi^n_{tt}$.}
Testing \eqref{Prop1Galerkin} with  $\psi^n_{tt} \in C([0,T];V_n)$, integrating with respect to time, and applying H\" older's inequality yields 
\begin{equation*}  
\begin{aligned}
&\| \psi_{tt}^n \|_{L^2L^2}^2 + \frac{b}{2} | \nabla \psi_t^n(t) |_{L^2}^2 
\\  \leq& \, \frac{b}{2} |\nabla \psi_1^n|_{L^2}^2+ \|\psi_{tt}^n\|_{L^2L^2} \big( \|\alpha\|_{L^\infty(\Omega \times (0,T))} \|\Delta \psi^n \|_{L^2L^2}  + \|\beta \|_{L^\infty L^4}\| \nabla \psi_t^n\|_{L^2L^4} + \|f \|_{L^2 L^2} \big),
\end{aligned}
\end{equation*} 
for all $t\in [0,T_n]$. Applying the Gagliardo-Nirenberg's inequality to the $\beta$-term and then Young's $\eps$-inequality then leads to
\begin{equation}  \label{Prop1Energy3}
\begin{aligned}
&(1-3\eps) \| \psi_{tt}^n \|_{L^2L^2}^2 + \frac{b}{2} | \nabla \psi_t^n(t) |_{L^2}^2 \\
\leq& \, \frac{b}{2} |\nabla \psi_1^n|_{L^2}^2+ C\|\alpha\|_{L^\infty(\Omega \times (0,T))}^2 \|\Delta \psi^n\|_{L^2L^2}^2 +\frac{1}{4\eps} \|f\|_{L^2L^2}^2 
+ \eps \|\Delta \psi_t^n\|_{L^2L^2}^2 \\
&+ C \|\beta\|_{L^\infty L^4}^\frac{2}{1-d/4} \|\nabla \psi_t^n\|_{L^2L^2}^2,
\end{aligned}
\end{equation} 
for all $t\in [0,T_n]$ and some $\eps>0$. However, we still cannot absorb the term $\|\Delta \psi_t^n\|_{L^2L^2}^2$ on the right-hand side in \eqref{Prop1Energy3} by any term on the left side of the estimates derived so far. We have to continue further by testing \eqref{Prop1Galerkin} with $-\Delta \psi^n_t(s), s\in[0,T_n]$. \medskip \\
%%%%%%%%%%%%%%%%%%%%%%%%%%%%%%%%%%%%%%%%%%%%%%%%%%%%%%%%%%%%%%%%%%%%%%%%%%%%%%%%%%%%%%%%%%%%55
\noindent \textit{Testing with $-\Delta \psi^n_{t}$.} We note that $\Delta \psi^n_t \in C([0,T_n];V_n)$ . Testing \eqref{Prop1Galerkin} with $-\Delta \psi^n_t(s) \in V_n, s\in [0,T_n]$, integrating with respect to time, and applying H\" older's inequality results in
\begin{equation*} 
\begin{aligned} & \,\frac12 |\nabla \psi_t^n(t)|_{L^2}^2+b\|\Delta \psi_t^n\|_{L^2L^2}^2 \\
 \leq& \, \frac12 |\nabla \psi_1^n|_{L^2}^2 + \|\Delta \psi_t^n\|_{L^2L^2} \big( \|\alpha\|_{L^\infty(\Omega \times (0,T))} \|\Delta \psi^n\|_{L^2L^2}  + \|\beta \|_{L^\infty L^4} \|\nabla \psi_t^n \|_{L^2L^4}  +\|f \|_{L^2 L^2} \big),\end{aligned} 
\end{equation*}
for all $t\in [0,T_n]$. As before, we apply Young's $\eps$-inequality to obtain
\begin{equation} \label{Prop1Energy4}
\begin{aligned}
& \,\frac12 |\nabla \psi_t^n(t)|_{L^2}^2+(b-3\eps)\|\Delta \psi_t^n\|_{L^2L^2}^2
\\\leq& \frac12 |\nabla \psi_1^n|_{L^2}^2  + C \|\alpha\|_{L^\infty(\Omega \times (0,T))}^2 \|\Delta \psi^n\|_{L^2L^2}^2 + \frac{1}{4\eps}\|f\|_{L^2L^2}^2 + C \|\beta\|_{L^\infty L^4}^2 \|\nabla \psi_t^n\|_{L^2L^2}^2,
\end{aligned} 
\end{equation} 
for all $t\in [0,T_n]$ and a conveniently chosen $\eps>0$. \\
\indent Now we can add the derived estimates \eqref{Prop1Energy1}-\eqref{Prop1Energy4}. Then by additionally bringing the remaining $\eps$-terms to the left-hand side we get
$$\begin{aligned}
&| \psi^n(t) |^2_{H^2} +| \psi^n_t (t)|^2_{ H^1} +\|\psi^n_t\|_{L^2H^2}^2+\|\psi^n_{tt}\|_{L^2L^2}^2 
\\ \lesssim& \,   |\psi^n_0|_{H^2}^2 +|\psi^n_1|_{H^1}^2 +\|f\|_{L^2L^2}^2 + \|\psi^n\|_{L^2H^2}^2(1+\|\alpha\|_{L^\infty(\Omega \times (0,T))}+\|\alpha\|_{L^\infty(\Omega \times (0,T))}^2 + \|\beta\|_{L^\infty L^4}^2)\\
&+\|\psi_t^n\|_{L^2 H^1}^2 \Big( \|\beta\|_{L^\infty L^4}^2+\|\beta\|_{L^\infty L^4}^\frac{2}{1-d/4}\Big),
\end{aligned}$$
for all $t \in [0,T_n]$. We then employ Gronwall's inequality \eqref{Gronwall_inequality} to find that there is a constant $C>0$, independent of $n$, such that
\begin{equation} \label{Prop1LowEnergy}
\begin{aligned} 
&\| \psi^n \|^2_{L^\infty H^2} +\| \psi^n_t \|^2_{L^\infty H^1} +\|\psi^n_t\|_{L^2H^2}^2+\|\psi^n_{tt}\|_{L^2L^2}^2 
\\\leq& \, C \exp\Big( C T \big(1+\|\alpha\|_{L^\infty(\Omega \times (0,T))}+\|\alpha\|_{L^\infty(\Omega \times (0,T))}^2 + \|\beta\|_{L^\infty L^4}^2 + \|\beta\|_{L^\infty L^4}^\frac{2}{1-d/4}\big) \Big) \\
&\, \times \big(|\psi_0|_{H^2}^2 + |\psi_1|_{H^1}^2 + \|f\|_{L^2L^2}^2\big). 
\end{aligned}
\end{equation}
Note that above we used \eqref{Prop1GalerkinIC} to uniformly bound $|\psi_0^n|_{H^2}^2$ and $|\psi_1^n|_{H^1}^2$. The estimate \eqref{Prop1LowEnergy} is independent of $T_n$ and consequently there is no blow-up in finite time, which allows us to extend the existence interval by setting $T_n=T$ for all $n\in \mathbb{N}$. \medskip \\
%%%%%%%%%%%%%%%%%%%%%%%%% Convergence of approx solutions %%%%%%%%%%%%%%%%%%%%%%%%%%%
\noindent \textbf{Convergence of approximate solutions.} Our next goal is to prove that a subsequence of $\{\psi^n\}_{n \in \mathbb{N}}$ converges to a solution of the linearized Blackstock equation. Due to the derived uniform estimate \eqref{Prop1LowEnergy} and Banach-Alaoglu's theorem, there exists a subsequence that we again denote by $\{\psi^n\}_{n \in \mathbb{N}}$ and a function $\psi: \Omega \times (0,T) \to \mathbb{R}$ such that  
\begin{alignat*}{4}
\psi^n  &\longweak \psi &&\text{ weakly-$\star$}  &&\text{ in } &&L^\infty(0,T;H_0^1\cap H^2),  \\
\psi_t^n &\longweak \psi_t &&\text{ weakly-$\star$} &&\text{ in } &&L^\infty(0,T;H^1_0),\\
\psi_t^n &\longweak \psi_t &&\text{ weakly} &&\text{ in } &&L^2(0,T;H_0^1\cap H^2),\\
\psi_{tt}^n &\longweak \psi_{tt} &&\text{ weakly} &&\text{ in } &&L^2(0,T;L^2), 
\end{alignat*}
as $n \to \infty$. Moreover, due to the compact embeddings  \cite[Corollary 4]{simon1986compact}
$$\begin{aligned}  W^{1,\infty,2}(0,T;H_0^1 \cap H^2,H_0^1) &\hookrightarrow \hookrightarrow C([0,T];H_0^1), \\ W^{1,\infty,2}(0,T;H_0^1,L^2) &\hookrightarrow \hookrightarrow C([0,T];L^2),
\end{aligned}$$
we even know that
\begin{equation} \label{strong_conv}
\begin{aligned}
\psi^n &\longrightarrow \psi &&\text{ strongly in} &&C([0,T];H_0^1), \\
\psi^n_t &\longrightarrow \psi_t &&\text{ strongly in} &&C([0,T];L^2),
\end{aligned}
\end{equation}
as $n \to \infty$. Next we want to show that the limit $\psi$ solves the initial-boundary value problem \eqref{IBVP_linearizedBlackstock} for the linearized Blackstock equation. \\
\indent Coming back to the Galerkin problem \eqref{Prop1Galerkin}, we test it with an arbitrary function $\eta \in C_c^\infty(0,T)$ to obtain
\begin{equation} \label{interim_eq}
\begin{aligned}
-\int_0^T  ( \psi^n_{t}(t), w_j)_{L^2} \eta'(t)\textup{d}t  =\int_0^T &\Big( (\alpha(t) \Delta \psi^n(t),w_j)_{L^2} +b (\Delta \psi^n_t(t), w_j)_{L^2} \\
& +(\beta(t) \cdot \nabla \psi_t^n(t),w_j)_{L^2} + ( f(t),w_j)_{L^2} \Big) \eta(t) \, \textup{d}t,
\end{aligned}
\end{equation}
for all $j \in \{1,...,n\}$. Note that the functional 
$$\begin{aligned}
 z \mapsto \int_0^T (z_t(t),w_j)_{L^2} \eta'(t) +(\alpha(t) \Delta z(t),w_j) \eta(t)+ b (\Delta z_t(t), w_j )_{L^2} \eta(t)  +(\beta(t)\cdot \nabla z_t(t),w_j\big)_{L^2} \eta(t) \,\textup{d}t
\end{aligned}$$
is linear and continuous on $H^1(0,T;H_0^1\cap H^2)$. Thus letting $n \to \infty$ in  \eqref{interim_eq} leads to
\begin{equation*} 
\begin{aligned}
 -\int_0^T  (\psi_{t}(t), w_j)_{L^2} \eta'(t) \textup{d}t = \int_0^T &\Big( (\alpha(t) \Delta \psi(t),w_j)_{L^2} +b (\Delta \psi_t(t),  w_j)_{L^2}\\
 & +(\beta(t) \cdot \nabla \psi_t(t),w_j)_{L^2} + ( f(t),w_j)_{L^2} \Big) \eta(t) \,\textup{d}t,
\end{aligned}
\end{equation*}
for all $j \in \mathbb{N}$ and all $\eta \in C_c^\infty(0,T)$. By construction $\cup_{n}V_n= \text{span}\{w_1,w_2,...\}$ is dense in $H_0^1 \cap H^2$ and therefore in $L^2$. We can then conclude that $\psi$ solves the equation
$$\psi_{tt} - \alpha \Delta \psi - b \Delta \psi_t = \beta \cdot \nabla \psi_t + f \quad \text{ in } L^2(0,T;L^2).$$
\noindent Due to the strong convergences stated in \eqref{strong_conv}, we know that 
\begin{alignat*}{2}
 \psi^n(0) &\longrightarrow \psi(0) &&\text{ in } H_0^1,\\
 \psi_t^n(0) &\longrightarrow \psi_t(0) &&\text{ in } L^2.
\end{alignat*}
Together with \eqref{Prop1GalerkinIC}, it follows that $\psi(0)=\psi_0$ and $\psi_t(0)=\psi_1$. Thus, $\psi$ is indeed a solution to \eqref{IBVP_linearizedBlackstock}. \medskip \\
\noindent \textbf{Uniqueness and the energy estimate.} Taking the limit inferior of the derived estimate \eqref{Prop1HighEnergy} for $\psi^n$ provides us with the desired energy estimate \eqref{Prop1Energy} for $\psi$ by exploiting the fact that the norms are weakly lower semicontinuous.\\
\indent Uniqueness follows by the linearity of the partial differential equation: assuming that we have two solutions $\psi^{(1)}$ and $\psi^{(2)}$, we can directly insert $\psi=\psi^{(1)}-\psi^{(2)}$ into the energy estimate \eqref{Prop1Energy} with $\psi_0=\psi_1=f=0$ to obtain $\psi=0$. \\
\indent Finally, we note that continuity in time of the solution follows due to the continuous embeddings
\begin{alignat*}{3}
\psi &\in H^1(0,T;H_0^1\cap H^2) &&\hookrightarrow && \, C([0,T];H_0^1 \cap H^2), \\
\psi_t &\in W^{1,2,2}(0,T;H_0^1 \cap H^2, L^2) &&\hookrightarrow && \, C([0,T];H_0^1);
\end{alignat*}
see \cite[Ch.1, Theorem 3.1]{lions1972non}.
\end{proof}
%%%%%%%%%%%%%%%%%%%%%%%%%%%%%%%%%%%%%% Proposition 2 %%%%%%%%%%%%%%%%%%%%%%%%%%%%%%%%
\subsection{Higher regularity of the solution} To show the well-posedness of the nonlinear model by using a fixed-point approach, the regularity of $\psi$ obtained in Proposition \ref{Proposition1} is not sufficient in a general two- and three-dimensional setting. Up to now we cannot expect that $\psi_t$ is bounded in $L^\infty(\Omega \times (0,T))$, because according to Proposition \ref{Proposition1} we only know that $\psi_t \in L^\infty(0,T;H^1)$. Therefore, we impose more regularity on the initial conditions.
\begin{proposition} \label{Proposition2}
	Let $\Omega \subset \mathbb{R}^d$, where $d \in \{1,2,3\}$, be a bounded, open, and $C^{2,1}$ regular domain and let $b>0$. Furthermore, assume that
	\begin{itemize}
	    \item $\alpha \in  W^{1,\infty,2}(0,T;W^{1,4},L^4)$, \smallskip 
	 %   \item there exist $\underline{\alpha}_0, \overline{\alpha}_0$ such that $0 \leq \underline{\alpha}_0 \leq \alpha (x,t) \leq \overline{\alpha}_0$ a.e. in $\Omega \times (0,T)$, \smallskip 
	    \item $\beta \in W^{1,\infty, \infty}(0,T;W^{1,4}(\Omega,\R^d),L^4(\Omega,\R^d)),$ \smallskip
	    \item $f \in W^{1,2,2}(0,T;H^1,L^2),$ \smallskip
	    \item $(\psi_0,\psi_1) \in (H_0^1 \cap H^3,H_0^1\cap H^2).$
	\end{itemize}
    Then for every $T>0$, the initial-boundary value problem for the linearized Blackstock equation \eqref{IBVP_linearizedBlackstock} admits a unique solution $\psi$ in the $C([0,T];L^2)$ sense. Moreover, the solution satisfies
	$$\begin{cases}
	\psi \in C([0,T];H_0^1\cap H^3), \\
	\psi_t \in C([0,T];H_0^1\cap H^2) \cap  L^2(0,T;H^3), \\
	\psi_{tt} \in C([0,T]; L^2) \cap L^2(0,T;H_0^1),
	\end{cases}$$ 
and the following energy estimate holds
\begin{equation} \label{Prop2Energy}
	\begin{aligned}
	& \| \psi \|^2_{L^\infty H^3} +\| \psi_t \|^2_{L^\infty H^2} +\|\psi_t\|_{L^2H^3}^2+\|\psi_{tt}\|_{L^\infty L^2}^2+\|\psi_{tt}\|_{L^2 H^1}^2 \\
	\leq& \, C_{\textup{P}_2}(T,\alpha,\beta) \big( |\psi_0|_{H^3}^2+|\psi_1|_{H^2}^2 +\|f\|_{L^2H^1}^2+\|f_t\|_{L^2L^2}^2+|f(0)|_{L^2}^2\big). 
	\end{aligned}
	\end{equation}
\end{proposition}
\begin{proof}
The main idea of the proof is to derive new energy estimates which will allow us to extend the regularity of $\psi$. To this end, we differentiate the Galerkin equation \eqref{Prop1Galerkin} with respect to time and the linearized Blackstock equation \eqref{IBVP_linearizedBlackstock} with respect to space. \\
\indent Proposition~\ref{Proposition1} already provides us with a unique solution. However, now the coefficients and the source term in the Galerkin equation \eqref{Prop1Galerkin} are all continuous in time due to the embeddings
\begin{equation*} 
\begin{aligned}
\alpha &\in W^{1,\infty,2}(0,T;W^{1,4},L^4) &&\hookrightarrow C([0,T];L^4), \\
\beta &\in W^{1,\infty, \infty}(0,T;W^{1,4}(\Omega,\R^d),L^4(\Omega,\R^d)) &&\hookrightarrow C([0,T]; L^4(\Omega,\R^d)), \\
f &\in H^1(0,T;L^2) &&\hookrightarrow C([0,T];L^2).
\end{aligned}
\end{equation*}
Recall that we rewrote the Galerkin problem as the initial-value problem \eqref{Prop1GalerkinODE}. In the present setting, the right-hand side of the equation is continuous with respect to $t$ for a fixed $r$. We are thus allowed to apply the Cauchy-Peano existence theorem to infer that \eqref{Prop1GalerkinODE} has a local solution $u^n \in C^1([0,T_n];\mathbb{R}^{2d})$. By definition of $u^n$, we have $\xi^n \in C^2([0,T_n];\mathbb{R}^d)$ and by construction of the approximate solution it follows that $\psi^n \in C^2([0,T_n];V_n)$. We can extend the time interval to $[0,T]$ by employing the same argument as in Proposition~\ref{Proposition1}. \smallskip \\ 
%%%%%%%%%%%%%%%%%%%%%%%%%%%%%%%%%%%%%%%%%%%%%%%%%%%%%%%%%%%%%%%%%%%%%%%%%%%%%%%%%%%%%%%%%%%%%%%%%%%%%%%%%%%%%%%%%%%%
\noindent \textit{Bootstrap argument to show $\psi^n \in H^3(0, T; V_n)$.} To take the time derivative of the Galerkin equation \eqref{Prop1Galerkin}, we first have to justify that $\psi_{ttt}^n$ is indeed well-defined. We observe that with the assumptions of Proposition~\ref{Proposition2}, we can show that $K^n \in H^1(0,T;\mathbb{R}^{n \times n})$ and $C^n \in H^1(0,T;\mathbb{R}^{n \times n})$. Indeed, it holds that
$$\begin{aligned}
\|K_{ij}\|_{H^1(0,T)}^2 &\lesssim \|\alpha\|_{L^2 L^2}^2 \|w_i\|_{H^2}^2 \|w_j\|_{H^1}^2+\|\alpha_t\|_{L^2 L^2}^2 \|w_i\|_{H^2}^2 \|w_j\|_{H^2}^2 \\&\quad +\|f\|_{L^2L^2}^2 \|w_j\|_{L^2}^2+\|f_t\|_{L^2L^2}^2 \|w_j\|_{L^2}^2, \\
\|C_{ij}\|_{H^1(0,T)}^2 &\lesssim  b \|w_i\|_{H^2}^2 \|w_j\|_{L^2}^2 + \|\beta\|_{L^2 L^2}^2 \|w_i\|_{H^2}^2 \|w_j\|_{H^1}^2 + \|\beta_t\|_{L^2 L^2}^2 \|w_i\|_{H^2}^2 \|w_j\|_{H^1}^2,
\end{aligned}$$
for all $i,j \in [1,n]$. Thus we have $$K^n \xi^n + C^n \xi_t^n \in H^1(0,T;\R^n), \ F^n \in H^1(0,T;\R^n) $$ and by \eqref{Prop1GalerkinMatrix_old} the same is valid for $\xi_{tt}^n$. In particular, $\xi_{ttt}^n \in L^2(0,T;\R^n)$ and hence $\psi^n_{ttt} \in L^2(0,T;V_n)$. \smallskip \\
%%%%%%%%%%%%%%%%%%%%%%%%%%%%%%%%%%%  \psi_{tt}(0) %%%%%%%%%%%%%%%%%%%%%%%%%%%%%%%%%%%%%%%%%%%%%%%%%%%%%%%%
\noindent \textbf{Higher-order energy estimate for $\psi^n$.} Before proceeding further, we show that we can bound $|\psi_{tt}^n(0)|_{L^2}$ with the data, which we need to obtain the final estimate. \noindent As we have seen, testing the Galerkin equation \eqref{Prop1Galerkin} with $\psi_{tt}^n(t) \in V_n$, $t\in[0,T]$, yields 
$$\begin{aligned} |\psi^n_{tt}(t)|_{L^2}^2 =-\big( &b\Delta \psi^n_t(t)+\alpha(t) \Delta \psi^n(t)+ \beta(t) \cdot \nabla \psi_t^n(t) +f(t),\psi^n_{tt}(t) \big)_{L^2},
\end{aligned}$$
for all $t\in [0,T]$. At $t=0$ we have
\begin{equation} \label{Prop1ICtwice}
\begin{aligned} |\psi^n_{tt}(0)|_{L^2} \leq \, b |\Delta \psi^n_1|_{L^2}+|\alpha(0)|_{L^4} |\Delta \psi^n_0|_{L^4} + |\beta(0)|_{L^4} |\nabla \psi^n_1|_{L^4} +|f(0)|_{L^2}.
\end{aligned} 
\end{equation}
%%%%%%%%%%%%%%%%%%%%%%%%%%%%%%%%%  Derivative w.r.t. time %%%%%%%%%%%%%%%%%%%%%%%%%%%%%%%%%%%%%%%%%%%%%%%%
\noindent \textit{Differentiated Galerkin system.}
We know that $\psi^n \in C^2([0,T];V_n) \cap H^3(0,T;V_n)$ which means that we are allowed to take the weak time derivative of the Galerkin equation \eqref{Prop1Galerkin}. Taking the derivative results in 
\begin{equation} \label{Prop1DiffTimeGalerkin}
\begin{aligned}
& (\psi_{ttt}^n,v)_{L^2} -(\alpha_t \Delta \psi^n,v)_{L^2} - (\alpha \Delta \psi_t^n,v)_{L^2} - b (\Delta \psi_{tt}^n, v)_{L^2} \\ &= (\beta_t \cdot \nabla \psi_t^n,v)_{L^2}+(\beta \cdot \nabla \psi_{tt}^n,v)_{L^2} + (f_t, v)_{L^2}, \\
\end{aligned} 
\end{equation} 
for all $v \in V_n$, pointwise a.e. in $[0,T]$. \smallskip \\
\noindent \textit{Testing with $ \psi^n_{tt}$.} 
Testing the equation \eqref{Prop1DiffTimeGalerkin} with $\psi_{tt}^n(t) \in V_n$, integrating with respect to time, employing H\" older's inequality, and taking the essential supremum in time yields
\begin{equation*} 
\begin{aligned} 
& \frac12 \|\psi^n_{tt}\|^2_{L^\infty L^2}+ b \|\nabla \psi^n_{tt}\|^2_{L^2 L^2} \\
\leq &  \, \frac12|\psi^n_{tt}(0)|^2_{L^2}+\|\psi^n_{tt}\|_{L^2 L^2} \big( \|\alpha\|_{L^\infty(\Omega \times (0,T))} \|\Delta \psi_t^n\|_{L^2L^2} +\| f_t\|_{L^2 L^2}  \big)
\\ &+\|\psi_{tt}^n\|_{L^2L^4} \big(\|\alpha_t\|_{L^2L^4} \|\Delta \psi^n\|_{L^\infty L^2}+ \|\beta_t\|_{L^2L^4} \|\nabla \psi_t^n\|_{L^\infty L^2}+\| \beta\|_{L^\infty L^4} \|\nabla \psi_{tt}^n\|_{L^2 L^2}\big).
\end{aligned}  
\end{equation*} 
By further utilizing Poincaré's, Gagliardo-Nirenberg's and Young's $\eps$-inequality, we obtain
\begin{equation*} 
\begin{aligned} 
& \frac12 \|\psi^n_{tt}\|^2_{L^\infty L^2}+ (b-5\eps) \|\nabla \psi^n_{tt}\|^2_{L^2 L^2} \\
\leq &  \, \frac12|\psi^n_{tt}(0)|^2_{L^2}+ \frac{K_\text{P}^2}{4\eps} \big( \|\alpha\|_{L^\infty(\Omega \times (0,T))}^2 \|\Delta \psi_t^n\|_{L^2L^2}^2  +\| f_t\|_{L^2 L^2}^2  \big)
\\ &+ \frac{K_1^2}{4\eps} \big( \|\alpha_t\|_{L^2L^4}^2 \|\Delta \psi^n\|_{L^\infty L^2}^2+ \|\beta_t\|_{L^2L^4}^2 \|\nabla \psi_t^n\|_{L^\infty L^2}^2 \big) +C\| \beta\|_{L^\infty L^4}^\frac{2}{1-d/4} \|\psi_{tt}^n\|_{L^2 L^2}^2.
\end{aligned}  
\end{equation*} 
 Next we make use of the bound for $|\psi_{tt}^n(0)|_{L^2}$ we derived in \eqref{Prop1ICtwice} and the energy estimate \eqref{Prop1LowEnergy} we derived in the proof of Proposition \ref{Proposition1} to bound $\|\Delta \psi^n\|_{L^\infty L^2}^2$, $\|\Delta \psi_t^n\|_{L^2 L^2}^2$, $\|\nabla \psi_t^n\|_{L^\infty L^2}^2$, and $\|\psi_{tt}^n\|_{L^2L^2}^2$. In this way we find 
\begin{equation*} 
%\label{Prop2_some_est}
\begin{aligned} 
& \|\psi^n_{tt}\|^2_{L^\infty L^2}+  \|\psi^n_{tt}\|^2_{L^2 H^1} \\
\lesssim & \,  |\alpha(0)|_{L^4}^2 |\psi^n_0|_{H^3}^2 + (1+|\beta(0)|_{L^4}^2) |\psi^n_1|_{H^2}^2 +|f(0)|_{L^2}^2  +  \|f_t\|_{L^2L^2}^2\\
&+ C_{\text{P}_1}(T,\alpha,\beta) \big(|\psi_0^n|_{H^2}^2+|\psi_1^n|_{H^1}^2+\|f\|_{L^2L^2}^2 \big) \big( \|\alpha\|_{L^\infty(\Omega \times (0,T))}^2+\|\alpha_t\|_{L^2 L^4}^2 + \|\beta\|_{L^\infty L^4}^\frac{2}{1-d/4}\\
&+ \|\beta_t\|_{L^2 L^4}^2 \big).
\end{aligned}  
\end{equation*} 
We add  this estimate to the lower energy estimate \eqref{Prop1LowEnergy} we obtained in Proposition \ref{Proposition1} to get the higher energy estimate
\begin{equation} \label{Prop1HighEnergy} 
\begin{aligned}
&\| \psi^n \|^2_{L^\infty H^2} +\| \psi^n_t \|^2_{L^\infty H^1} +\|\psi^n_t\|_{L^2H^2}^2+\|\psi_{tt}^n\|_{L^\infty L^2}^2+\|\psi_{tt}^n\|_{L^2H^1}^2 
\\ & \leq \, C_{1}(T, \alpha, \beta) \big( |\psi_0|_{H^3}^2+|\psi_1|^2_{H^2}+ \|f\|_{L^2L^2}^2+\|f_t\|_{L^2L^2}^2+|f(0)|_{L^2}^2\big).
\end{aligned}
\end{equation}
The constant above is given by 
$$\begin{aligned} C_{1}(T,\alpha,\beta) &= C_{\text{P}_1}(T,\alpha,\beta) \big(1+\|\alpha\|_{L^\infty(\Omega \times (0,T))}^2+\|\alpha_t\|_{L^2 L^4}^2 + \|\beta\|_{L^\infty L^4}^\frac{2}{1-d/4} + \|\beta_t\|_{L^2 L^4}^2  \big) \\ &\quad + C \big(1+|\alpha(0)|_{L^4}^2+|\beta(0)|_{L^4}^2 \big),
\end{aligned}$$
where $C_{\textup{P}_1}(T,\alpha,\beta)$ denotes the constant given by Proposition \ref{Proposition1}.\\
\indent We can then proceed as in Proposition~\ref{Proposition1} via weak limits to obtain a solution $\psi$ and then prove its uniqueness and continuous dependence on the data. \smallskip \\
\noindent \textbf{Higher regularity in space.} We now want to show $H^3$-regularity in space of the solution. We already know by Proposition \ref{Proposition1} that there is a $\psi$ such that
\begin{equation} \label{Prop2SpaceDiffReg}
    \frac{\textup{d}}{\textup{d}t} \Delta \psi+ \frac{\alpha}{b} \Delta \psi  =\frac{\psi_{tt}- \beta \cdot \nabla \psi_t- f}{b}
\end{equation} 
in the $L^2(0,T;L^2)$ sense. Note that the right-hand side belongs to $L^2(0,T;H^1)$ thanks to the regularity assumptions on $\beta$, $f$, and the regularity of $\psi$ given by Proposition \ref{Proposition1}. From the equation \eqref{Prop2SpaceDiffReg}, we find that
$$ \Delta \psi(t) = \displaystyle e^{-\int_0^t \frac{\alpha(\tau)}{b} \textup{d}\tau} \left(\Delta \psi_0 + \int_0^t e^{\int_0^s \frac{\alpha(\tau)}{b} \textup{d}\tau} \frac{\psi_{tt}(s)-\beta(s) \cdot \nabla \psi_t(s)- f(s)}{b} \, \textup{d}s \right)$$
for all $t\in [0,T]$, by continuity of $\Delta \psi$ and the Lebesgue integral. By a regularity result for elliptic partial differential equations \cite[Theorem 2.5.1.1]{grisvard2011elliptic}, $\psi \in C([0,T];H_0^1 \cap H^3)$ and hence by \eqref{Prop2SpaceDiffReg} also $\psi_t \in L^2(0,T;H_0^1\cap H^3)$. As a consequence, the linearized Blackstock equation is satisfied in $L^2(0,T;H^1)$. Therefore, we are allowed to take the gradient of the equation. We conclude that $\psi$ satisfies the vector-valued system
\begin{equation} \label{Prop2DiffSpace}
\nabla \psi_{tt} - \nabla \alpha \Delta \psi - \alpha \nabla \Delta \psi -b \nabla \Delta \psi_t= \nabla \beta  \nabla \psi_t+\beta \Delta \psi_t + \nabla f
\end{equation}
in the $L^2(0,T;L^2(\Omega,\mathbb{R}^d))$ sense. \smallskip \\ 
\noindent \textbf{Energy estimate.} Testing the equation \eqref{Prop2DiffSpace} with $-\chi_{(0,t)} \nabla\Delta \psi_t \in L^2(0,T;L^2(\Omega,\mathbb{R}^d))$ yields, after employing the typical inequalities, 
\begin{equation} \label{Prop2Energy1}
\begin{aligned} 
&(b/2-5\eps) \|\nabla\Delta \psi_{t}\|^2_{L^2L^2} \\
\leq& \, \frac12 |\Delta \psi_1|_{L^2}^2 + C\|\alpha \|_{L^\infty W^{1,4}}^2 \big(\|\Delta \psi\|_{L^2L^2}^2+ \|\nabla \Delta \psi\|_{L^2L^2}^2\big) +C \|\beta\|_{L^\infty W^{1,4}}^2 \|\Delta \psi_t\|_{L^2L^2}^2 \\
&+ C\|\nabla f\|_{L^2 L^2}^2+\frac{1}{2b} \|\nabla \psi_{tt}\|_{L^2L^2}.
\end{aligned} 
\end{equation}
Above we have made use of 
$$\|\Delta \psi\|_{L^2L^4}^2 \leq K_5^2 \big( \|\Delta \psi\|_{L^2L^2}^2+\|\nabla \Delta \psi\|_{L^2L^2}^2\big),$$
as introduced in the theoretical preliminaries. Further testing \eqref{Prop2DiffSpace} with $-\chi_{(0,t)} \nabla \Delta \psi \in L^2(0,T;L^2(\Omega,\mathbb{R}^d))$ and using Young's $\eps$-inequality yields 
\begin{equation} \label{Prop2Energy2}
\begin{aligned} 
&\frac{b}{2} |\nabla\Delta \psi(t)|^2_{L^2} 
\\\leq& \, \frac{1}{2}\|\nabla \psi_{tt}\|_{L^2L^2}^2+\frac{1}{2}\|\nabla \Delta \psi\|^2_{L^2L^2}+C\|\nabla \alpha\|_{L^\infty L^4}^{2} \big( \|\Delta \psi\|_{L^2L^2}^2 + \|\nabla \Delta \psi\|_{L^2L^2}^2 \big) \\
&+\|\alpha\|_{L^\infty L^\infty}\|\nabla \Delta \psi\|^2_{L^2L^2}+\frac{b}{2} |\nabla \Delta \psi_0|_{L^2}^2+C(\|\beta\|_{L^\infty W^{1,4}}^2+1)\|\nabla \Delta \psi\|_{L^2L^2}^2+\frac12 \|\Delta \psi_t\|_{L^2L^2}^2\\
&+C \|\nabla f\|_{L^2 L^2}^2.
\end{aligned}
\end{equation}
%%%%%%%%%%%%%%%%%%%%%%%%%%%%%%%%%%%%%%%%%%%%%%%%%%%%%%%%%%%%%%%%%%%%%%%%%%%%%%
\noindent We then add \eqref{Prop2Energy2} to the continuous version of the estimate \eqref{Prop1HighEnergy} and apply the modification of Gronwall's inequality \eqref{Gronwall_inequality}  to get
\begin{equation} \label{higher_est}
\begin{aligned}
&\|\nabla \Delta \psi\|_{L^\infty L^2}^2  \\
\lesssim & \,   C_2(T, \alpha, \beta) \, \big( |\psi_0|_{H^3}^2+|\psi_1|^2_{H^2}+ \|f\|_{H^1L^2}^2+ \|f_t\|_{L^2L^2}^2+|f(0)|_{L^2}^2\big) ,
\end{aligned}
\end{equation}
where the constant is given by $$C_2(T, \alpha, \beta)= \, C_{1}(T, \alpha, \beta)\, \textup{exp}(\, CT(\|\nabla \alpha\|_{L^\infty L^4}^{2}+\|\alpha\|_{L^\infty L^\infty}+\|\beta\|_{L^\infty W^{1,4}}^2+1)).$$
\indent To obtain a bound on $\psi_t$ in $L^\infty(0,T; H^2)$, we employ a trick from \cite{MizohataUkai} and now consider the following initial boundary value problem
 \begin{align} \label{diff_problem}
\begin{cases}
v_{tt}-\alpha \Delta v-b \Delta v_t=h(x,t) \quad \text{in } \Omega \times (0,T),
\smallskip\\ v=0 \quad \text{ on }  \partial \Omega \times (0,T),  
\smallskip\\ (v, v_t)=(\psi_1, \, \alpha(0) \Delta \psi_0+b\Delta \psi_1+ \beta(0) \cdot \nabla \psi_1 +f(0)) \quad \text{ on }   \Omega \times \{t=0\},
\end{cases}
\end{align}
where the right side of the partial differential equation is given by $$h=\alpha_t \Delta \psi + \beta_t \cdot \nabla \psi_t+ \beta \cdot \nabla \psi_{tt}+f_t.$$ 
We can estimate $h$ as follows
\begin{equation}
\begin{aligned}
    \|h\|_{L^2L^2} \leq& \, \|\alpha_t\|_{L^2 L^4} \|\Delta \psi\|_{L^\infty L^4} +\| \beta_t\|_{L^\infty L^4} \|\nabla \psi_t\|_{L^2 L^4}\\
 & + \|\beta\|_{L^\infty L^\infty}\|\nabla \psi_{tt}\|_{L^2 L^2}+\|f_t\|_{L^2L^2}.
\end{aligned}
\end{equation}
By employing the embedding $H_0^1 \cap H^3 \hookrightarrow H_0^1 \cap W^{2,4}$ from the preliminaries, we further get
\begin{equation*}
\begin{aligned}
\|h\|^2_{L^2L^2} \lesssim C_3(T, \alpha, \beta)\big( |\psi_0|_{H^3}^2+|\psi_1|^2_{H^2}+ \|f\|_{L^2H^1}^2+\|f_t\|^2_{L^2L^2}+|f(0)|_{L^2}^2\big),
\end{aligned}
\end{equation*}
where 
\begin{align*}
    C_3(T, \alpha, \beta)=\|\alpha_t\|^2_{L^2 L^4}C_2(T, \alpha, \beta)+C_1(T, \alpha, \beta)\big(\|\alpha_t\|^2_{L^2 L^4}+\|\beta\|^2_{L^\infty L^\infty}+\| \beta_t\|^2_{L^\infty L^4}\big)+C.
\end{align*}
Therefore, we know that $h \in L^2(0,T; L^2)$, $v(0) \in H_0^1 \cap H^2$, and $v_t(0) \in L^2$. It can be shown, similarly to Proposition~\ref{Proposition1}, that problem \eqref{diff_problem} has a unique solution such that
\begin{equation} \label{regularity_v}
\begin{cases} 
v \in C([0, T]; H_0^1) \cap C_w([0,T];H_0^1\cap H^2), \\
v_t \in C([0,T]; L^2) \cap L^2(0,T; H_0^1),\\
v_{tt} \in L^2(0,T;H^{-1}). 
\end{cases}
\end{equation}
Moreover, the following estimate holds
\begin{equation} \label{Prop1Energy_old}
\begin{aligned}
& \| v \|^2_{L^\infty H^2} +\| v_t \|^2_{L^\infty L^2} +\|\nabla v_t\|_{L^2L^2}^2 \\
\lesssim& \, \textup{exp}\big(CT(\|\alpha\|^2_{L^\infty L^\infty}+1)\big)\big(|v(0)|_{H^2}^2+ |v_t(0)|_{L^2}^2 +\|h\|_{L^2L^2}^2\big).
\end{aligned}
\end{equation}
By integrating the partial differential equation in \eqref{diff_problem} from $0$ to $t$, it can be shown that $\psi=\psi_0+\int_0^t v(x,s) \, \textup{ds}$. Therefore, from \eqref{regularity_v} we have that 
\begin{equation*} 
\begin{cases} 
\psi_{tt} \in C([0,T]; L^2) \cap L^2(0,T; H_0^1),\\
\psi_{ttt} \in L^2(0,T;H^{-1}). 
\end{cases}
\end{equation*}
Moreover, we obtain 
\begin{equation} \label{estimate_psi_t}
\begin{aligned}
 \| \psi_t \|^2_{L^\infty H^2} 
 \lesssim & \, C_4(T, \alpha, \beta)\big(|\psi_0|_{H^3}^2+ |\psi_1|^2_{H^2}+\|f\|^2_{L^2 H^1}+\|f_t\|^2_{L^2L^2}+|f(0)|_{L^2}^2\big), % \\
\end{aligned}
\end{equation}
where $C_4(T, \alpha, \beta)=(C+C_3(T, \alpha, \beta))\, \textup{exp}(CT(1+\|\alpha\|^2_{L^\infty L^\infty}))$. \\
\indent We can now add \eqref{Prop2Energy1} multiplied by $b>0$, \eqref{higher_est}, \eqref{estimate_psi_t}, and the continuous version of \eqref{Prop1HighEnergy} and apply Gronwall's inequality to get the final estimate \eqref{Prop2Energy}, where
% $$C_{\textup{P}_2}(T, \alpha, \beta)= \, C \displaystyle \bigg(\sum_{i=1}^4 C_i(T, \alpha, \beta)+1\bigg) \,\text{exp}\big(CT(\|\alpha\|^2_{L^\infty W^{1,4}}+\|\beta\|^2_{L^\infty W^{1,4}}) \big).$$
$$C_{\textup{P}_2}(T, \alpha, \beta)=(C_2+C_4)\,\text{exp}\big(CT(\|\alpha\|^2_{L^\infty W^{1,4}}+\|\beta\|^2_{L^\infty W^{1,4}}) \big).$$
 Finally, continuity in time of $\psi$ and $\psi_t$ follows due to the continuous embeddings
\begin{alignat*}{3}
 \psi &\in H^1(0,T;H_0^1\cap H^3) &&\hookrightarrow && \, C([0,T];H_0^1\cap H^3), \\
 \psi_t &\in W^{1,2,2}(0,T;H_0^1\cap H^3,H_0^1) && \hookrightarrow && \,  C([0,T];H_0^1\cap H^2); 
\end{alignat*}
see \cite[Ch.1, Theorem 3.1]{lions1972non}. 
\end{proof}
%%%%%%%%%%%%%%%%%%%%%%%%%%%%%%%% Nonlinear Blackstock %%%%%%%%%%%%%%%%%%%%%%%%%%%%%%%%%%%%%%%%
\section{Local well-posedness}\label{section:local_wellposedness}
%%%%%%%%%%%%%%%%%%%%%%%%%%%%%%%%%%%%%%%%%%%%%%%%%%%%%%%%%%%%%%%%%%%%%%%%%%%%%%%%
\indent We now return to the nonlinear Blackstock model and employ the Banach fixed point theorem to show existence and uniqueness of a solution for small data. \\
\indent We introduce the space
\begin{equation*}
\begin{aligned}
\mathcal{X}=\Bigl\{v:(0,T)\times \Omega \to \mathbb{R} \ \Big\vert & \,  v \in L^\infty(0,T;H_0^1\cap H^3),\\
&v_t \in L^\infty (0,T;H_0^1\cap H^2)\cap L^2(0,T;H^3), \\
& v_{tt} \in C([0,T]; L^2) \cap L^2(0,T;H_0^1) \, \Bigr\},
\end{aligned} 
\end{equation*}
equipped with the norm 
\begin{align*}
\|v\|_{\mathcal{X}} =\max\Big\{ \| v \|_{L^\infty H^3} , \|v_t\|_{L^\infty H^2}, \|v_{t}\|_{L^2H^3},\|v_{tt}\|_{C L^2}, \|v_{tt}\|_{L^2 H^1} \Big\}.
\end{align*}
We are now ready to state the main result.
\begin{theorem}[Local well-posedness]\label{Theorem_Local_wellp} Let $\Omega \subset \mathbb{R}^d$, where $d \in \{1,2,3\}$, be a bounded, open, and $C^{2,1}$ regular domain. Let $c^2, b>0$, and $k \in \mathbb{R}$. Assume that $(\psi_0,\psi_1) \in (H_0^1 \cap H^3,H_0^1\cap H^2)$ and that 
\begin{align*}
    |\psi_0|_{H^3}^2+|\psi_1|_{H^2}^2  \leq \kappa, \quad \kappa>0.
\end{align*}
Then for sufficiently small $\mathcal{\kappa}$ and final time $T>0$, there exists a unique solution $\psi \in \mathcal{X}$ of the initial-boundary value problem \eqref{Blackstock} for the Blackstock equation. The solution depends continuously on the initial data in the $\mathcal{X}$-norm.
\end{theorem}
\begin{proof} %We carry out the proof with the assumption that $k \neq 0$. In the simpler case when $k=0$, the proof follows analogously with minor modifications. \\ 
\indent Following \cite{KaltenbacherLasiecka_Westervelt, KaltenbacherLasiecka_Kuznetsov}, we employ the Banach fixed-point theorem. To this end we introduce an operator $$\mathcal{F}: v \mapsto \psi,$$ where
\begin{equation} \label{space_B}
\begin{aligned} 
v \in	\mathcal{B}=\Bigl\{v \in \mathcal{X} \ \Big\vert & \, (v, v_t)\vert_{t=0}=(\psi_0, \psi_1),  \ \|v\|_{\mathcal{X}} \leq M \Bigr\},
\end{aligned} 
\end{equation}
and $\psi$ is given as the solution in the $L^\infty(0,T;L^2)$ sense to
\begin{align}\label{eq_psi_v}
\psi_{tt}-c^2(1+k v_t)\Delta \psi-b \Delta \psi_t=2 \nabla v \cdot \nabla \psi_t
\end{align}
with $ (\psi, \psi_t)\vert_{t=0}=(\psi_0, \psi_1)$. The parameter $M>0$  in \eqref{space_B} will be chosen to guarantee that $\mathcal{F}$ is a contractive self-mapping. \smallskip \\
%%%%%%%%%%%%%%%%%%%%%%% B is closed %%%%%%%%%%%%%%%%%%%%%%%%%%%%%%%%%%%%%%%%%
\noindent\textbf{$\mathcal{B}$ is closed in the topology induced by $\|\cdot \|_{\mathcal{X}}$.} We first show that $\mathcal{B}$ is closed in the topology of $\mathcal{X}$, which will later allow us to apply the Banach fixed-point theorem on $\mathcal{F}$ to conclude that it has a unique fixed point. \\
\indent Let  $\{ v^n \}_{n\in \mathbb{N}} \subset \mathcal{B}$ be an arbitrary $\|\cdot\|_{\mathcal{X}}$-converging sequence and let $v=\displaystyle \lim_{n\to \infty} v^n$ in $\mathcal{X}$. We need to show that $v \in \mathcal{B}$. The typical embeddings yield
\begin{alignat*}{2} v^n &\longrightarrow v &&\text{ strongly in } C([0,T];H_0^1 \cap H^3), \\
v_t^n &\longrightarrow v_t &&\text{ strongly in } C([0,T];H_0^1\cap H^2)\cap L^2(0,T;H^3), \\
v_{tt}^n &\longrightarrow v_{tt} &&\text{ strongly in } C([0,T];L^2)\cap L^2(0,T;H^1_0), 
\end{alignat*} 
from which it follows that $v(0)=\psi_0$ and $v_t(0)=\psi_1$. Moreover, 
\begin{align*}
\|v\|_{\mathcal{X}} &\leq \liminf_{n \to \infty} \|v^n-v\|_{\mathcal{X}}+\|v^n\|_{\mathcal{X}} \leq M. %,\\
%\|v_t\|_{L^\infty L^\infty} &\leq \liminf_{n \to \infty} \|v_t^n-v_t\|_{L^\infty L^\infty}+\|v^n_t\|_{L^\infty L^\infty} 
%\\ &\leq K_2K_3 \liminf_{n \to \infty} \|v_n-v\|_\mathcal{X}+\|v^n_t\|_{L^\infty L^\infty}  \leq m \leq \frac{1}{|k|}.
\end{align*}
Therefore we can conclude that $v \in \mathcal{B}$. \smallskip \\
%%%%%%%%%%%%%%%%%%%%%%%%%%%%%%% F is a self-mapping %%%%%%%%%%%%%%%%%%%%%%%%%%%%%%%%%%%%%%%%%
\noindent \textbf{$\mathcal{F}$ is a self-mapping.} Next we want to show that $\psi=\mathcal{F}v \in \mathcal{B}$ for every $v \in \mathcal{B}$. The equation \eqref{eq_psi_v} fits into the framework of Proposition \ref{Proposition2} if we set
\begin{alignat*}{3}
&\alpha&&=c^2(1+kv_t) &&\in W^{1,\infty,2}(0,T; W^{1,4}, L^4), \\
&\beta&&=2\nabla v &&\in W^{1,\infty, \infty}(0,T; W^{1,4}(\Omega,\R^d), L^4(\Omega,\R^d)), \\
&f &&= 0 &&  \in W^{1,2,2}(0,T;H^1,L^2).
\end{alignat*}
According to Proposition~\ref{Proposition2}, it follows that $\psi \in \mathcal{X}$. Thanks to the energy estimate \eqref{Prop2Energy}, if the initial data is small in the sense of
\begin{equation*} 
\begin{aligned}
 \widetilde{C}(T, M)(   |\psi_0|_{H^3}^2+|\psi_1|_{H^2}^2 ) \leq M^2,
\end{aligned}
\end{equation*}
it also follows that \\
\begin{align*}
 \|\psi\|_{\mathcal{X}} \leq M.
\end{align*}
The constant $\widetilde{C}(T, M)$ is obtained from $C_{\textup{P}_2}(T, \alpha, \beta)$ by replacing the $\alpha$ and $\beta$ terms by their upper bounds which depend on $M$. We infer that $\psi \in \mathcal{B}$. \smallskip \\
%%%%%%%%%%%%%%%%%%%%%%%%%%%%%%% F is contractive %%%%%%%%%%%%%%%%%%%%%%%%%%%%%%%%%%%%%%%%%
\noindent	\textbf{$\mathcal{F}$ is contractive.} To show contractivity, we take $v^{(1)}, v^{(2)} \in \mathcal{B}$ and set $\psi^{(1)}=\mathcal{F} v^{(1)}$, $\psi^{(2)}=\mathcal{F} v^{(2)} \in \mathcal{B}$, and $v=v^{(1)}-v^{(2)}$. The difference $\psi=\psi^{(1)}-\psi^{(2)}$ then satisfies the equation
\begin{equation*}
\begin{aligned}
\psi_{tt}-c^2(1+kv_t^{(1)}) \Delta \psi-b \Delta \psi_t=c^2 k  \Delta \psi^{(2)} v_t +2 \nabla v^{(1)}\cdot \nabla \psi_t+2  \nabla \psi_{t}^{(2)}\cdot \nabla v 
\end{aligned}
\end{equation*}
in the $L^\infty(0,T;L^2)$ sense with $(\psi, \psi_t)\vert_{t=0}=(0,0)$. This equation also fits into the framework of Proposition \ref{Proposition2} if we choose
\begin{alignat*}{3}
&\alpha&&=c^2(1+kv_t^{(1)}) &&\in W^{1,\infty,2}(0,T; W^{1,4}, L^4), \\
&\beta&&=2\nabla v^{(1)} &&\in W^{1,\infty, \infty}(0,T; W^{1,4}(\Omega,\R^d), L^4(\Omega,\R^d)), \\
&f &&= c^2k \Delta \psi^{(2)}v_t + 2 \nabla \psi_t^{(2)}\cdot \nabla v &&  \in W^{1,2,2}(0,T;H^1,L^2).
\end{alignat*}
Thus Proposition \ref{Proposition2} provides the energy estimate
$$\begin{aligned}
	& \| \psi \|^2_{L^\infty H^3} +\| \psi_t\|^2_{L^\infty H^2} +\|\psi_t\|_{L^2H^3}^2 +\|\psi_{tt}\|_{L^\infty L^2}^2+\|\psi_{tt}\|_{L^2 H^1}^2 \\
	\leq& \, \widetilde{C}(T, M) \big( \|f\|_{L^2L^2}^2+\|\nabla f\|_{L^2L^2}^2+\|f_t\|_{L^2L^2}^2\big),
\end{aligned}$$
We further estimate the $f$-terms on the right-hand side by
$$\begin{aligned}
&\|f\|_{L^2L^2}^2+\|\nabla f\|_{L^2L^2}^2+\|f_t\|_{L^2L^2}^2 \\ 
\leq& \, c^4 k^2 \|\psi^{(2)}\|_{L^2H^3}^2 \|v_t\|_{L^\infty L^\infty}^2 + 4 K_2^4 \|\psi_t^{(2)}\|_{L^2H^2}^2 \|v\|_{L^\infty H^3}^2 +c^4 k^2 \|\psi^{(2)}\|_{L^2 H^3}^2 \|v_t\|_{L^\infty L^\infty}^2 
\\ &+ c^4 k^2 K_5^2 \|\psi^{(2)}\|_{L^2 H^3}^2 \|v_t\|_{L^\infty H^2}^2 
+4 K_4^2 K_5^2 \|\psi_t^{(2)}\|_{L^2L^2}^2 \|v\|_{L^\infty H^3}^2 + 4 K_2^2 K_5^2 \|\psi_t^{(2)}\|_{L^2 H^2}^2 \|v\|_{L^\infty H^3}^2
\\&+ c^4k^2 \|\psi_t^{(2)}\|_{L^2 H^2}^2 \|v_t\|_{L^\infty L^\infty}^2 
+c^4 k^2K_1^2 K_5^2 \|\psi^{(2)}\|_{L^\infty H^3}^2 \|v_{tt}\|_{L^2 H^1}^2+ 4K_4^2 K_5^2 \|\psi_{tt}^{(2)}\|_{L^2 H^1}^2 \|v\|_{L^\infty H^3}^2 
\\ &+ 4 K_2^4 \|\psi_t^{(2)}\|_{L^2 H^3}^2 \|v_t\|_{L^\infty H^2} ^2
\end{aligned}$$
\noindent We can additionally make use of the estimate
\begin{align*}
\|\psi^{(2)}\|_{L^2H^3}^2\leq T\|\psi^{(2)}\|_{\mathcal{X}}^2 
%\leq T M^2 \, \widetilde{C}(T, M) \, (   |\psi_0|_{H^3}^2+|\psi_1|_{H^2}^2  ), 
\leq T M^2,
\end{align*}
and proceed similarly for the other terms involving the $L^2$-norm in time of $\psi^{(2)}, \psi_t^{(2)},\psi_{tt}^{(2)}$. In this way, we obtain
$$\begin{aligned}
    \|\mathcal{F}v^{(1)}-\mathcal{F}v^{(2)}\|_{\mathcal{X}}^2 
    %\leq \, T \, \widetilde{C}(T, M) \, (   |\psi_0|_{H^3}^2+|\psi_1|_{H^2}^2  )\, \|v\|_{\mathcal{X}}^2.
    \leq \, C M^2 \, \widetilde{C}(T,M) \,  (T+1) \, \|v\|_{\mathcal{X}}^2
\end{aligned}$$
Altogether for small enough $M$ and final time $T$, we have
\begin{align*}
\|\mathcal{F}v^{(1)}-\mathcal{F}v^{(2)}\|_{\mathcal{X}} \leq q \|v^{(1)}-v^{(2)}\|_{\mathcal{X}} \quad \text{for every } v^{(1)}, v^{(2)} \in \mathcal{B}, \quad \text{where} \ 0< q <1. \smallskip
\end{align*}
\noindent \textbf{Existence and uniqueness.} Finally, the existence and uniqueness of a solution to the initial-boundary value problem \eqref{Blackstock} for the Blackstock equation follow from the Banach fixed-point theorem. \smallskip\\
\noindent \textbf{Continuous dependence on the data.} It remains to prove that the solution depends continuously on the data. We introduce $\psi$ as the unique fixed point of $\mathcal{F}$ that corresponds to initial data $(\psi_0,\psi_1)\in H_0^1 \cap H^3 \times H_0^1 \cap H^2$ and parameter $M$. Further, let $\phi$ be the unique fixed point to a different contraction mapping that corresponds to data $(\phi_0,\phi_1)\in H_0^1\cap H^3 \times H_0^1\cap H^2$ and parameter $M_\phi$. We employ the triangle inequality 
$$\|\psi-\phi\|_\mathcal{X} \leq \|\mathcal{F} \psi-\mathcal{F}\phi\|_\mathcal{X} +\|\mathcal{F}\phi-\phi\|_\mathcal{X},$$
where we can already estimate the first norm by the contractivity of $\mathcal{F}$. Note that $\mathcal{F}\phi$ is the solution of 
$$(\mathcal{F}\phi)_{tt}-c^2(1+k\phi_t)\Delta \mathcal{F}\phi -b\Delta (\mathcal{F}\phi)_t = 2 \nabla \phi \cdot \nabla (\mathcal{F}\phi)_t$$
with initial conditions $(\mathcal{F}\phi)(0)=\psi_0, (\mathcal{F}\phi)_t(0)=\psi_1$.  Let us introduce $\widetilde{\phi}= \mathcal{F}\phi-\phi$. It holds that
\begin{equation} \label{EqContDependenceProof} \widetilde{\phi}_{tt} - c^2(1+k\phi_t)\Delta \widetilde{\phi} - b \Delta \widetilde{\phi}_t =2\nabla \phi \cdot \nabla \widetilde{\phi}_t 
\end{equation}
with initial conditions $\widetilde{\phi}(0)=\psi_0-\phi_0, \widetilde{\phi}_t(0)=\psi_1-\phi_1$.\\
\indent We already know the regularity of $\phi$ as it is the unique local solution of the Blackstock equation corresponding to the initial conditions $(\phi_0,\phi_1)$. Thus, we can apply Proposition~\ref{Proposition2} on the partial differential equation \eqref{EqContDependenceProof}, where we set $\alpha=c^2(1+k\phi_t)$, $\beta=2 \nabla \phi$, and $f=0$. In this way we obtain existence and uniqueness of a solution $\widetilde{\phi}$ that satisfies the energy estimate
\begin{equation} \label{EqContDependenceProof2} \|\mathcal{F}\phi - \phi \|_\mathcal{X}=\|\widetilde{\phi}\|_\mathcal{X} \leq  \widetilde{C}(T, M_\phi) \big(|\psi_0-\phi_0|_{H^3} + |\psi_1-\phi_1|_{H^2} \big). 
\end{equation} 
In order to prove continuous dependence on the data, we now employ the contractivity of $\mathcal{F}$ and the estimate \eqref{EqContDependenceProof2} in the following way
$$\begin{aligned}
 \|\psi-\phi\|_\mathcal{X} \leq& \, \|\mathcal{F}\psi-\mathcal{F} \phi\|_\mathcal{X} + \|\mathcal{F} \phi-\phi\|_X \\
 \leq& \, q\|\psi-\phi\|_\mathcal{X} + \widetilde{C}(T, M_\phi) \big(|\psi_0-\phi_0|_{H^3}+|\psi_1-\phi_1|_{H^2}\big), \quad 0<q<1.
\end{aligned}$$
Bringing the $q$-term to the left-hand side yields
$$\|\psi-\phi\|_\mathcal{X} \leq \frac{ \widetilde{C}(T, M_\phi)}{1-q}  \big(|\psi_0-\phi_0|_{H^3}+|\psi_1-\phi_1|_{H^2}\big), \quad 0< q <1,$$
which completes the proof.
\end{proof}
%%%%%%%%%%%%%%%%%%%%%%%%%%%%%%%% Global well-posedness %%%%%%%%%%%%%%%%%%%%%%%%%%%%%%%%%%%%%%%%
\section{Global well-posedness and energy decay}\label{section:global_wellposedness}
To obtain global well-posedness, we follow the general approach taken in \cite{KaltenbacherLasiecka_Westervelt, KaltenbacherLasiecka_Kuznetsov} and revisit the energy estimates from Propositions \ref{Proposition1} and \ref{Proposition2} for the nonlinear Blackstock equation to get an estimate independent of final time. We restrict ourselves now to the \textit{non-degenerate} case and look for a solution in the space
\begin{equation*} 
\begin{aligned} 
\hat{\mathcal{B}}=\Bigl\{v \in \mathcal{X} \ \Big\vert & \, (v, v_t)\vert_{t=0}=(\psi_0, \psi_1), \ \|v_t\|_{L^\infty(\Omega \times (0,T))}\leq m < 1/|k|,  \ \|v\|_{\mathcal{X}} \leq M \Bigr\},
\end{aligned} 
\end{equation*}
when $k \neq 0$. It can be shown analogously to Theorem~\ref{Theorem_Local_wellp} that a unique local-in-time solution exists in this space. If $k=0$, then we set $\hat{\mathcal{B}}=\mathcal{B}$.\\
\indent We introduce the energy as
\begin{equation} \mathcal{E}[\psi](t)=|\psi(t)|_{H^3}^2 + |\psi_t(t)|_{H^2}^2+|\psi_{tt}(t)|_{L^2}^2, \quad \text{for }\psi \in \mathcal{X}, \label{GlobalEnergy} \end{equation}
and the corresponding initial energy as $$\mathcal{E}[\psi](0)=|\psi_0|_{H^3}^2+|\psi_1|_{H^2}^2+| c^2(1+k \psi_1) \Delta \psi_0 + b \Delta \psi_1 + 2 \nabla \psi_0 \cdot \nabla \psi_1 |^2_{L^2}.$$
\begin{theorem}[Global well-posedness] \label{TheoremGlobal} Let $\Omega \subset \mathbb{R}^d$, where $d \in \{1,2,3\}$, be a bounded, open, and $C^{2,1}$ regular domain. Let $c^2>0$ and $k \in \mathbb{R}$. Furthermore, let $(\psi_0,\psi_1) \in (H_0^1 \cap H^3,H_0^1\cap H^2)$. Then there exists a $\kappa>0$, such that for $b$  sufficiently large, the solution of the initial-boundary value problem \eqref{Blackstock} corresponding to the initial data $\mathcal{E}[\psi](0)\leq \kappa$ is global in time. In other words, there exist positive constants $\kappa$ and $\mathcal{M}$, such that as long as $\mathcal{E}[\psi](0)\leq \kappa$, then 
\begin{align*}
\mathcal{E}[\psi](t) \leq\mathcal{M} \quad \text{for all} \ t\in [0,\infty).
\end{align*}
\end{theorem}
\begin{proof} We follow a similar procedure as in the proofs of Propositions \ref{Proposition1} and \ref{Proposition2}, but this time we avoid Gronwall's lemma to get a bound independent of $T$, where $T$ corresponds to the maximal time that guarantees local well-posedness in $\hat{\mathcal{B}}$ of the initial-boundary value problem \eqref{Blackstock} for the Blackstock equation. By utilizing this time-independent estimate, we are able to extend the local solution of Theorem \ref{Theorem_Local_wellp} to the global time domain $[0,\infty)$. We begin by testing the Blackstock equation \eqref{Blackstock} and its differentiated version with different functions. \smallskip  \\
\noindent \textit{Testing with $-\chi_{(0,t)} \Delta \psi$.} Since we know that $\Delta \psi \in L^2(0,T;L^2)$, we can test the Blackstock equation \eqref{Blackstock} with $-\chi_{(0,t)} \Delta \psi$. This action yields
\begin{equation} \label{est1_global_wellp}
    \begin{aligned}
     &(c^2-2\eps) \int_0^t |\Delta \psi(s)|_{L^2}^2 \, \textup{d}s + \frac{b}{2} |\Delta \psi(t)|^2_{L^2}
    \\ \leq& \, \frac{b}{2} |\Delta \psi_0|_{L^2}^2+ \frac{K_\textup{P}^2}{4\eps} \int_0^t |\nabla \psi_{tt}(s)|_{L^2}^2\, \textup{d}s+\frac{K_2^4}{\eps} \esssup_{s \in (0,T)} |\Delta \psi(s)|_{L^2}^2\int_0^t |\Delta \psi_t(s)|_{L^2}^2 \, \textup{d}s \\ &+ c^2|k| \esssup_{s \in (0,T)} |\psi_t(s)|_{L^\infty} \int_0^t |\Delta \psi(s)|_{L^2}^2 \, \textup{d}s,
    \end{aligned} 
\end{equation} 
for all $t\in [0,T]$, where $\varepsilon >0$. \smallskip  \\
\noindent \textit{Testing with $-\chi_{(0,t)}\Delta \psi_{t}$.} Since $\Delta \psi_t \in L^2(0,T;L^2)$, we next take $-\chi_{(0,t)}\Delta \psi_t$ as a test function in \eqref{Blackstock}.  In this way we obtain the estimate
\begin{equation} \label{est2_global_wellp}
\begin{aligned} &\frac12 |\nabla \psi_t(t)|_{L^2}^2+\frac{c^2}{2} |\Delta \psi(t)|_{L^2}^2+b\int_0^t |\Delta \psi_t(s)|_{L^2}^2 \, \textup{d}s \\
  \leq& \, \frac12 |\nabla \psi_1|_{L^2}^2  +\frac{c^2}{2} |\Delta \psi_0|_{L^2}^2  + \frac{K_2^4}{\eps} \esssup_{s\in (0,T)} |\Delta \psi_t(s)|_{L^2}^2 \int_0^t |\Delta \psi_t(s)|_{L^2}^2\, \textup{d}s \\
  &+\frac{c^4k^2}{4\eps} \esssup_{s \in (0,T)} |\psi_t(s)|_{L^\infty}^2 \int_0^t |\Delta \psi_t(s)|_{L^2}^2 \, \textup{d}s+2\eps \int_0^t |\Delta \psi(s)|_{L^2}^2 \, \textup{d}s ,
\end{aligned} 
\end{equation} 
for all $t\in [0,T]$, where $\eps>0$. \smallskip  \\
\noindent \textit{Testing the time-differentiated equation with $ \chi_{(0,t)}\psi_{tt}$.} Since the time-differentiated version of the Blackstock equation is satisfied in the $L^2(0,T;H^{-1})$ sense, we can test it with functions in $L^2(0,T;H_0^1)$. Testing with $\chi_{(0,t)} \psi_{tt} \in L^2(0,T;H_0^1)$ yields
\begin{equation}
\begin{aligned} 
& \frac12 |\psi_{tt}(t)|^2_{L^2}+\frac{c^2}{2} |\nabla \psi_t(t)|_{L^2}^2+ (b-4\eps) \int_0^t  |\nabla \psi_{tt}(s)|^2_{L^2}\, \textup{d}s  \\
\leq & \, \frac12|\psi_{tt}(0)|^2_{L^2}+\frac{c^2}{2} |\nabla \psi_1|^2_{L^2}+ \frac{c^4k^2 K_1^4+4K_1^2 K_2^2}{4\eps}\esssup_{s\in (0,T)} |\Delta \psi(s)|_{L^2}^2 \int_0^t |\nabla \psi_{tt}(s)|_{L^2}^2\, \textup{d}s
\\ &+ \frac{K_2^4 K_\textup{P}^2}{\eps} \esssup_{s\in (0,T)} |\Delta \psi_t(s)|_{L^2}^2  \int_0^t |\Delta \psi_t(s)|_{L^2}^2 \, \textup{d}s  +\frac{c^4k^2K_\text{P}^2}{4\eps} \esssup_{s\in (0,T)} | \psi_t(s)|_{L^\infty }^2  \int_0^t |\Delta \psi_t(s)|_{L^2}^2 \, \textup{d}s,
\end{aligned}  
\end{equation} 
 for all $t\in [0,T]$ and some $\eps>0$. \smallskip \\
 \noindent \textit{Testing the time-differentiated equation with $ -\chi_{(0,t)} \Delta \psi_{t}$.} Note that we are not allowed to directly test the time-differentiated equation with $ -\chi_{(0,t)} \Delta \psi_{t}$ since $\Delta \psi_{t} \notin L^2(0,T;H_0^{1})$. However, we can employ a similar trick as before and look at the following initial-boundary value problem
\begin{align*} %\label{diff_problem_global}
\begin{cases}
v_{tt}-c^2 \Delta v-b \Delta v_t=h(x,t) \quad \text{in } \Omega \times (0,T),
\smallskip\\ v=0 \quad \text{ on }  \partial \Omega \times (0,T),  
\smallskip\\ (v, v_t)=(\psi_1, \, c^2(1+k\psi_1) \Delta \psi_0+b\Delta \psi_1+ 2\nabla \psi_0 \cdot \nabla \psi_1) \quad \text{ on }   \Omega \times \{t=0\},
\end{cases}
\end{align*}
where the right side of the partial differential equation is given by
\begin{equation*}
\begin{aligned}
h=c^2 k (\psi_{t}\Delta \psi_{t}+\psi_{tt}\Delta \psi)+2( \nabla \psi \cdot \nabla \psi_{tt} +|\nabla \psi_t|^2) \in L^2(0,T; L^2).
\end{aligned}
\end{equation*}
It can be shown by employing the Galerkin approximation that if $b$ is sufficiently large it holds that
\begin{equation*}
\begin{aligned}
    \frac{b}{4}\|\Delta v\|^2_{L^\infty L^2}+(c^2-\eps)\|\Delta v\|^2_{L^2 L^2} \leq \frac{b+1}{2}|\Delta v(0)|^2_{L^2}+\frac{c^2}{2}|\nabla v(0)|^2_{L^2}+|v_t(0)|^2_{L^2}+\frac{1}{2\eps} \|h\|^2_{L^2L^2},
\end{aligned}
\end{equation*}
first in a discretized setting and then via weak limits also in the continuous one. By noting that $\psi=\psi_0 + \int_0^t v(x,s)\, \textup{d}s$, from here we further obtain
\begin{equation} \label{est3_global_wellp}
\begin{aligned}
  &\frac{b}{4}|\Delta \psi_t(t)|^2_{L^2} +(c^2-\eps)\int_0^t |\Delta \psi_t(s)|^2_{L^2}\, \textup{d}s \\
  \leq& \, \frac{b+1}{2}|\Delta \psi_1|^2_{L^2}+\frac{c^2}{2}|\nabla \psi_1|^2+|\psi_{tt}(0)|^2_{L^2}\\
  &+\frac{c^4k^2}{2\eps}\bigg(\esssup_{s \in (0,T)}|\psi_t(s)|^2_{L^\infty}\int_0^t |\Delta \psi_t(s)|_{L^2}^2 \, \textup{d}s+ K^2_1 \esssup_{s \in (0,T)} |\Delta \psi (s)|^2_{L^4}\int_0^t |\nabla \psi_{tt}(s)|^2_{L^2}\, \textup{d}s \bigg)\\
  &+\frac{2}{\eps}\bigg(K_4^2\esssup_{s \in (0,T)} |\Delta \psi (s)|^2_{L^4}\int_0^t |\nabla \psi_{tt}(s)|^2_{L^2}\, \textup{d}s+K^4_{2} \esssup_{s \in (0,T)} |\Delta \psi_t (s)|^2_{L^2}\int_0^t |\Delta \psi_{t}(s)|^2_{L^2}\, \textup{d}s\bigg),
\end{aligned}
\end{equation}
for all $t \in [0,T]$. Here $K_4<\infty$ denotes the embedding constant from \eqref{embed_constants} and by employing the embedding $H_0^1 \cap H^3 \hookrightarrow H_0^1 \cap W^{2,4}$ we can make the term
$$\esssup_{s \in (0,T)} |\Delta \psi(s)|_{L^4}^2 \leq K_5^2 \esssup_{s \in (0,T)} |\psi(s)|_{H^3}^2  \leq K_5^2 M^2$$
sufficiently small.\\
\noindent \textit{Testing the space-differentiated equation with $-\chi_{(0,t)}\nabla \Delta \psi$.} Note that the solution provided by Theorem \ref{Theorem_Local_wellp} satisfies the Blackstock equation in the $L^2(0,T;H_0^1)$ sense. Thus it is well-defined to test the space-differentiated version with test functions in $L^2(0,T;L^2(\Omega,\mathbb{R}^d))$.
Taking $-\chi_{(0,t)}\nabla \Delta \psi \in L^2(0,T;L^2(\Omega,\mathbb{R}^d))$  as a test function in the space-differentiated Blackstock equation gives us 
\begin{equation} \label{est4_global_wellp}
\begin{aligned} &(c^2-3\eps) \int_0^t|\nabla \Delta \psi(s)|_{L^2}^2 \, \textup{d}s+\frac{b}{2} |\nabla\Delta \psi(t)|^2_{L^2} 
\\ \leq& \, \frac{b}{2} |\nabla \Delta \psi_0|_{L^2}^2 + \frac{1}{4\eps} \int_0^t|\nabla \psi_{tt}(s)|_{L^2}^2\, \textup{d}s + c^2|k| \esssup_{s \in (0,T)} |\psi_t(s)|_{L^\infty} \int_0^t |\nabla \Delta \psi(s)|_{L^2}^2 \, \text{d}s\\
&+ \esssup_{s\in (0,T)}|\Delta \psi_t(s)|_{L^2} \big(K_2 K_5(2+c^2|k|)+2 K_4 K_5  \big) \int_0^t|\nabla \Delta \psi(s)|_{L^2}^2 \, \textup{d}s \\
& + \frac{4K_4^2 K_5^2+K_2^2K_5^2(4+c^4k^2)}{4\eps}  \esssup_{s \in (0,T)} |\Delta \psi_t(s)|_{L^2}^2 \int_0^t |\Delta \psi(s)|^2 \, \text{d}s,
\end{aligned}
\end{equation}
for all $t\in [0,T]$. After adding the  derived upper bounds \eqref{est1_global_wellp}-\eqref{est4_global_wellp}, we obtain 
\begin{equation} \label{GlobalEstimate1}
    \begin{aligned}
        &\frac{b}{2} |\nabla \Delta \psi(t)|_{L^2}^2 +\frac{b+c^2}{2} |\Delta \psi(t)|_{L^2}^2+ \frac{b}{4} |\Delta \psi_t(t)|_{L^2}^2 + \frac12 |\psi_{tt}(t)|_{L^2}^2 + C_1\int_0^t|\Delta \psi(s)|_{L^2}^2 \, \textup{d}s
        \\
        & +C_2 \int_0^t|\Delta \psi_t(s) |_{L^2}^2\, \textup{d}s + C_3\int_0^t |\nabla \psi_{tt}(s)|_{L^2}^2\, \textup{d}s 
        + C_4 \int_0^t |\nabla \Delta \psi(s)|_{L^2}^2 \, \textup{d}s
        \\ \leq& \, \frac12 \max\big\{b+c^2,3+b+2c^2\big\} \big( |\psi_0|_{H^3}^2 + |\psi_1|_{H^2}^2+|\psi_{tt}(0)|_{L^2}^2 \big),
    \end{aligned}
\end{equation}
for all $t\in [0,T]$. The constants appearing above are given by
\begin{equation} \label{global_wellp_constants}
\begin{aligned}
  C_1 =& \, c^2-4\eps-c^2|k| \|\psi_t\|_{L^\infty L^\infty}- \|\Delta \psi_t\|_{L^\infty L^2}^2 \frac{4K_4^2K_5^2+K_2^2 K_5^2 (4+c^4k^2)}{4\eps} ,
\\  C_2 =& \, b+c^2-\eps -\|\Delta \psi\|_{L^\infty L^2}^2 \frac{K_2^4}{\eps} -\|\Delta \psi_t\|_{L^\infty L^2}^2 \frac{3K_2^4+K_2^4 K_{\text{P}}^2}{\eps}  -\|\psi_t\|_{L^\infty L^\infty}^2 \frac{c^4k^2(3+K_{\text{P}}^2)}{4\eps},
\\  C_3 =& \, b-4\eps- \frac{1+K_\textup{P}^2}{4 \eps} - \|\Delta \psi\|_{L^\infty L^2}^2 \frac{c^4k^2 K_1^4 + 4K_1^2 K_2^2}{4\eps}  - \|\Delta \psi\|_{L^\infty L^4}^2 \frac{c^4k^2K_1^2+2K_4^2}{2\eps}, 
%\\  C_4 &= b - \frac{4 K_2^2 K_4^2+4 K_4^2 K_5^2+c^4k^2 K_2^2 K_3^2 +c^4 k^2 K_2^2 K_4^2 }{4\eps}\|\Delta \psi_t\|_{L^\infty L^2}^2 ,
\\  C_4 =& \,  c^2 - 3\eps - c^2|k|\|\psi_t\|_{L^\infty L^\infty} - \|\Delta \psi_t\|_{L^\infty L^2} \big(K_2K_5(2+c^2|k|)+2K_4K_5\big), \\
\end{aligned}
\end{equation}
We note that the assumption of non-degeneracy is crucial to ensure that $C_1>0$. For small enough $\|\psi\|_\mathcal{X}$ and large enough $b$, we can guarantee that $C_2,...,C_4$ are positive constants as well.  \\ 
\indent By using the energy function $\mathcal{E}[\psi]$ we introduced in \eqref{GlobalEnergy}, we can rewrite \eqref{GlobalEstimate1} as
\begin{equation}
\begin{aligned}
& \min\big\{2,b\big\}\  \mathcal{E}[\psi](t) +2 \min\big\{C_1+C_4,C_2, K_\text{P}^{-2}\, C_3\big\} \int_0^t \mathcal{E}[\psi](s) \, \textup{d}s \\
\leq& \, C \max\big\{b+c^2,3+b+2c^2\big\}  \ \mathcal{E}[\psi](0) \label{GlobalEnergyEstimate2}
\end{aligned}
\end{equation}
for all $t \in [0,T]$. From here it follows that
\begin{align*} 
\min\big\{2,b\big\} \ \mathcal{E}[\psi](t) \leq \, C \max\big\{b+c^2,3+b+2c^2\big\} \ \mathcal{E}[\psi](0) 
\end{align*}
for all $t \in [0, T]$, where the constant $C$ is independent of $T$. We can then derive a similar estimate for $t \in [T, 2T]$
\begin{align*} 
\min\big\{2,b\big\} \ \mathcal{E}[\psi](t) \leq \, C \max\big\{b+c^2,3+b+2c^2\big\} \ \mathcal{E}[\psi](T). 
\end{align*}
Successively repeating this action yields the energy inequality \eqref{GlobalEnergyEstimate2} for all $t\in [0,\infty)$ and thus the global well-posedness of the initial-boundary value problem \eqref{Blackstock} for the Blackstock equation follows.
\end{proof}
\noindent We are next interested in obtaining a rate for the decay of the energy \eqref{GlobalEnergy}. 
\begin{theorem}[Energy decay] \label{TheoremExpDecay} Let the assumptions of Theorem~\ref{TheoremGlobal} hold. Then there exists a constant $\omega>0$ such that the energy $\mathcal{E}[\psi](\cdot)$ decays exponentially with decay rate $\omega$, that means
$$\mathcal{E}[\psi](t) \leq Ce^{-\omega t}\mathcal{E}[\psi](0).$$
\end{theorem}
\begin{proof}
We employ the same estimates in the proof of Theorem~\ref{TheoremGlobal}, but instead of taking $\chi_{(0,t)}$ as a prefactor for the test functions, we now choose $\chi_{(s,t)}$ for $0 \leq s\leq t<\infty$. For instance, testing the Blackstock equation with $-\chi_{(s,t)} \Delta \psi$ gives an estimate similar to \eqref{est1_global_wellp}
\begin{equation*} 
    \begin{aligned} 
     &(c^2-2\eps) \int_s^t |\Delta \psi(\tau)|_{L^2}^2 \, \textup{d}\tau + \frac{b}{2} |\Delta \psi(t)|^2_{L^2}
    \\ \leq& \, \frac{b}{2} |\Delta \psi(s)|_{L^2}^2+ \frac{K_\textup{P}^2}{4\eps} \int_s^t |\nabla \psi_{tt}(\tau)|_{L^2}^2\, \textup{d}\tau +\frac{K_2^4}{\eps} \esssup_{\tau \in (0,\infty)} |\Delta \psi(\tau)|_{L^2}^2\int_s^t |\Delta \psi_t(\tau)|_{L^2}^2 \, \textup{d}\tau
    \\ &+ c^2|k| \esssup_{\tau \in (0,\infty)} |\psi_t(\tau)|_{L^\infty}^2 \int_s^t |\Delta \psi(\tau)|_{L^2}^2 \, \textup{d}\tau,
    \end{aligned} 
\end{equation*} 
We proceed in the same way for the other estimates. In this way, we arrive at the energy inequality
\begin{align*}
& \min\big\{2,b \big\} \mathcal{E}[\psi](t)+ 2\min\big\{C_1+C_4,C_2, K_\text{P}^{-2} \, C_3\big\} \int_s^t \mathcal{E}[\psi](\tau)\, \textup{d}\tau \\
\leq& \,  C \max\big\{2b+c^2,3+b+2c^2\big\} \mathcal{E}[\psi](s),
\end{align*}
where $C_1,...,C_4$ are the constants from \eqref{global_wellp_constants} with the difference that the essential supremum in the norms is taken over the global time domain $(0,\infty)$.
From here we first conclude that $0\leq E[\psi](t)\leq C E[\psi](s)$ for $s \leq t$
%, which means $E[\psi](\cdot)$ is non-increasing and non-negative. 
and letting $t \to \infty$ yields the inequality
$$  \int_s^\infty \mathcal{E}[\psi](\tau)\, \textup{d}\tau \leq C\frac{\max\{b+c^2,3+b+2 c^2\}}{\min\{C_1+C_4,C_2,K_\text{P}^{-2} \, C_3\}} \mathcal{E}[\psi](s)$$
for all $s \in [0,\infty)$. Applying a modification of the Haraux-Lagnese inequality \cite[Theorem 1.5.9]{qin2016analytic} provides us with the desired exponential decay estimate
$$\mathcal{E}[\psi](s) \leq Ce^{-\omega t} \mathcal{E}[\psi](0),$$
where the decay rate is given by 
$$\omega = C\ \frac{\min\{C_1+C_4,C_2,K_\text{P}^{-2} \, C_3\}}{\max\{b+c^2,3+b+2 c^2\}}.$$
\end{proof}
\noindent We can see that the decay rate goes to zero in the case of overdamping when $b \rightarrow \infty$.
\subsection{The one-dimensional case}
In the one-dimensional case when $\Omega \subset \mathbb{R}$, the situation simplifies considerably since we can make use of the embedding $$\psi_t \in L^\infty(0,T; H_0^1) \hookrightarrow L^\infty(0,T; L^\infty).$$
We can employ the fixed point approach relying only on Proposition~\ref{Proposition1} and show local well-posedness analogously to Theorem~\ref{Theorem_Local_wellp}. Moreover, we can employ the same techniques as in Theorem~\ref{TheoremGlobal} and Theorem~\ref{TheoremExpDecay} to obtain global well-posedness and the decay of the energy
 \begin{align*} 
 E(t)=|\psi(t)|_{H^2}^2+|\psi_t(t)|_{H^1}^2.
\end{align*}
 We state here the result in the one-dimensional case without proof.
\begin{theorem}\label{theorem:1D} Let $\Omega \subset \mathbb{R}$ be a finite interval and let $b$, $c^2>0$, and $k \in \mathbb{R}$. Assume that $(\psi_0,\psi_1) \in (H_0^1\cap H^2, H_0^1)$ and that $$|\psi_0|_{H^2}^2+|\psi_1|_{H^1}^2 \leq \kappa, \quad \kappa>0.$$
Then for sufficiently small $\kappa$ and final time $T>0$, there exists a unique solution $\psi$ of the initial-boundary value problem \eqref{Blackstock} for the Blackstock equations with the regularity
$$\begin{cases} \psi \in C([0,T];H_0^1\cap H^2), \\ \psi_t \in C([0,T];H_0^1)\cap L^2(0,T;H^2), \\ \psi_{tt} \in L^2(0,T;L^2). \end{cases}$$
The solutions depends continuously on the initial data. Furthermore, for sufficiently large $b$, the solution is global in time and its energy decays exponentially with rate $\omega>0$, that means
\begin{align*} 
E(t) \leq C e^{-\omega t} E(0).
\end{align*}
\end{theorem}
%%%%%%%%%%%%%%%%%%%%%%%%%%%%%%%% Numerical treatment %%%%%%%%%%%%%%%%%%%%%%%%%%%%%%%%%%%%%%%%
\section{Numerical treatment of the Blackstock equation}\label{section:numerical_treatment}
\indent We now focus our attention on the numerical treatment of the Blackstock equation. We follow the well-explored numerical strategies for nonlinear sound propagation~\cite{Hoffelner, kagawa1992finite, manfred, muhr2017isogeometric, walsh2007finite}, where finite elements are employed in space and the Newmark or the Generalized-$\alpha$ scheme is used for time stepping. In particular, we employ B-splines as basis functions within the framework of Isogeometric Analysis (IGA); see~\cite{cottrell2009isogeometric, hughes2005isogeometric}. The use of B-splines as basis functions in nonlinear acoustics can be found in earlier works on Burgers' equation; see for example~\cite{davies1978application,zhu2009numerical}. \\
\indent By approximating the acoustic velocity potential as
\begin{align*}
    \psi \approx \displaystyle \sum_{i=1}^{n_\textup{eq}} N_i \, \psi_i,
\end{align*}
and using the same shape functions for the test space, we arrive at the following semi-discrete form
\begin{equation} \label{Semidiscrete_weak_form}
\begin{aligned}
\displaystyle  \sum_{j=1}^{n_\textup{eq}} \Bigl \{& \int_\Omega N_i N_j \, \textup{d}x \, \ddot{\psi}_j+c^2 \int_\Omega \nabla N_i \cdot \nabla N_j \, \textup{d}x \, \psi_j +b \int_\Omega \nabla N_i \cdot \nabla N_j \, \textup{d}x \, \dot{\psi}_j  \\
&+c^2k\int_\Omega (\nabla N_i \cdot \nabla N_j)\Bigl(\sum_{l=1}^{n_\textup{eq}} N_l \dot{\psi_l}\Bigr) \, \textup{d}x \, \psi_j \\
&-(2-c^2k)\int_\Omega   N_i \, \Bigl (\nabla N_j \cdot \nabla \sum_{l=1}^{n_\textup{eq}} N_l \dot{\psi_l}\Bigr)\, \textup{d}x \, \psi_j \Bigr \}=0,
\end{aligned}
\end{equation}
for all $t \in (0,T]$. If we denote by $\underline{w}=[w_1 \ldots w_{n_\textup{eq}}]$ the representation of a function $w$ in the basis $\{N_i\}_{i=1}^{n_\textup{eq}}$, then the initial conditions should be satisfied in the sense of
\begin{align} \label{discrete_initial}
(\underline{\psi}, \underline{\dot{\psi}})\vert_{t=0}=(\underline{\psi}_0, \underline{\dot{\psi}}_0).
\end{align} We can rewrite \eqref{Semidiscrete_weak_form}, \eqref{discrete_initial} in the matrix form as follows
\begin{equation} \label{matrix_equation}
\begin{aligned}
\begin{cases}
\mathbf{M} \ddot{\underline{\psi}}+\mathbf{K}\underline{\psi}+\mathbf{C} \dot{\underline{\psi}}+ \bm{\mathcal{T}}(\dot{\underline{\psi}})\, \underline{\psi} - \bm{\mathcal{N}}(\dot{\underline{\psi}})\, \underline{\psi}=\bm{0},\\
\underline{\psi}(0)=\underline{\psi}_0, \\
\dot{\underline{\psi}}(0)=\underline{\dot{\psi}}_0,
\end{cases}
\end{aligned}
\end{equation}
where the mass, stiffness, and damping matrices are given by
\begin{equation*}
\begin{aligned}
& \mathbf{M}= \bigwedge\limits_{e=1}^{n_e} [m_{pq}^e], \quad m_{p q}^e= \int_{\Omega^e} N_{p}N_{q} \,  \textup{d}x, \\
& \mathbf{K}= \bigwedge\limits_{e=1}^{n_e} [\kappa_{pq}^e], \quad \kappa_{p q}^e= c^2 \int_{\Omega^e} \nabla N_{p} \cdot \nabla N_{q} \,  \textup{d}x, \\
& \mathbf{C}= \bigwedge\limits_{e=1}^{n_e} [c_{pq}^e], \quad c_{p q}^e= b \int_{\Omega^e} \nabla N_{p} \cdot \nabla N_{q} \,  \textup{d}x, \end{aligned}
\end{equation*}
where $n_e$ denotes the number of elements $\Omega^e$.
The tensors can be computed via
\begin{equation*}
\begin{aligned}
& \bm{\mathcal{T}}(\dot{\underline{\psi}})= \bigwedge\limits_{e=1}^{n_e} [T_{pq}^e], \quad T_{pq}^e= c^2k \, \int_{\Omega^e} (\nabla N_p \cdot \nabla N_{q}) \Bigl(\sum_{r=1}^{n_\textup{eq}} N_r \dot{\psi}_r \Bigr) \,  \textup{d}x,\\
& \bm{\mathcal{N}}(\dot{\underline{\psi}})= \bigwedge\limits_{e=1}^{n_e} [N_{pq}^e], \quad N_{pq}^e= (2-c^2k) \, \int_{\Omega^e} N_p \, \nabla N_{q} \cdot \Bigl(\nabla \sum_{r=1}^{n_\textup{eq}} N_r \dot{\psi}_r \Bigr) \,  \textup{d}x.
\end{aligned}
\end{equation*}
\subsection{Choice of the time-integration scheme}\indent In nonlinear ultrasound propagation, the number of higher harmonics to the fundamental wave frequency grows with the distance from the source. To resolve higher harmonics, a fine spatial mesh and time discretization are typically needed. Moreover, we have to be aware of Gibbs' phenomenon, which occurs when a discontinuous function is approximated by smooth functions and causes spurious oscillations on the wave peaks. For this reason, we use the Generalized-$\alpha$ method \cite{chung1993time} for the time discretization. The scheme allows us to control the numerical dissipation that is introduced to higher harmonics, while minimally affecting the lower frequencies; see also the discussion in \cite{muhr2017isogeometric} for the Westervelt equation.\\ 
\indent Let $0=t_0< \ldots <t_m=T$ be a given partition of the time interval $[0,T]$. To simplify exposition, we assume a uniform grid with time step $\Delta t$ and define $\psi_n=\psi(\cdot, n \Delta t)$, $n\in [0, m]$. Applying the Generalized $\alpha$-scheme results in the system
%\enlargethispage{1cm}
\begin{equation*} 
\begin{aligned}
&\mathbf{M} \ddot{\underline{\psi}}_{n+1-\alpha_m}+\mathbf{C} \dot{\underline{\psi}}_{n+1-\alpha_f}+\mathbf{K} \underline{\psi}_{n+1-\alpha_f} \\
&+\bm{\mathcal{T}}(\dot{\underline{\psi}}_{n+1-\alpha_f})\underline{\psi}_{n+1-\alpha_f}-\bm{\mathcal{N}}(\dot{\underline{\psi}}_{n+1-\alpha_f})\underline{\psi}_{n+1-\alpha_f}= \,\bm{0},
\end{aligned}
\end{equation*}
where the mid-point values are defined as follows
\begin{align*}
& \ddot{\underline{\psi}}_{n+1-\alpha_m}=(1-\alpha_m)\ddot{\underline{\psi}}_{n+1}+\alpha_m \underline{\ddot{\psi}}_n, \\
& \dot{\underline{\psi}}_{n+1-\alpha_f}=(1-\alpha_f)\dot{\underline{\psi}}_{n+1}+\alpha_f \underline{\dot{\psi}}_n,\\
& \underline{\psi}_{n+1-\alpha_f}=(1-\alpha_f)\underline{\psi}_{n+1}+\alpha_f \underline{\psi}_n,
\end{align*}
for $n\in [0, m-1]$.  We then use the classical Newmark scheme \cite{newmark1959method} with parameters $\beta$ and $\gamma$:
\begin{align*}
\underline{\psi}_{n+1}&=\underline{\psi}_{n}+\Delta t \, \underline{\dot{\psi}}_{n+1}+\frac{1}{2}\Delta t^2((1-2\beta)\underline{\ddot{\psi}}_{n}+2\beta\underline{\ddot{\psi}}_{n+1}), \\
\underline{\dot{\psi}}_{n+1}&=\underline{\dot{\psi}}_{n}+\Delta t((1-\gamma)\underline{\ddot{\psi}}_{n}+\gamma\underline{\ddot{\psi}}_{n+1}).
\end{align*}
The scheme is realized through predictor-corrector steps. At this point, we also introduce the effective mass matrix
$$\mathbf{\overline{M}}=(1-\alpha_m)\mathbf{M}+\gamma (1-\alpha_f)\Delta t \mathbf{C}+ \beta (1-\alpha_f)\Delta t^2 \mathbf{K};$$
see also \cite[Section 2.5]{manfred}. Since the Blackstock equation is nonlinear, we employ an iterative scheme to solve it. \\ \\
\begin{algorithm}[H]
 \SetAlgoLined
\caption{Numerical solver for the Blackstock equation with initial data} \label{algorithm_Blackstock} \vspace{2mm}
\textbf{Input}: initial data  $\underline{\psi}_0$, $\dot{\underline{\psi}}_0$ \vspace{2mm} \\
$t=0$, $n=0$\vspace{2mm}\\
$\mathbf{M} \, \ddot{\underline{\psi}}_n=-\mathbf{C} \dot{\underline{\psi}}_n-\mathbf{K} \underline{\psi}_n -\bm{\mathcal{T}}(\dot{\underline{\psi}}_n)\, \underline{\psi}_n + \bm{\mathcal{N}}(\dot{\underline{\psi}}_n)\underline{\psi}_n$\vspace{2mm}\\
\While{$t\leq T$ \vspace{1mm}}{ 
\text{perform predictor step} \vspace{1mm} \\
$\underline{\psi}_{\textup{pred}}=\underline{\psi}_{n}+\Delta t \dot{\underline{\psi}}_n+\displaystyle \frac{\Delta t^2}{2}(1-2\beta)\ddot{\underline{\psi}}_n$\vspace{2mm}\\
$\dot{\underline{\psi}}_{\textup{pred}}=\dot{\underline{\psi}}_n+(1-\gamma)\Delta t \ddot{\underline{\psi}}_n$\vspace{2mm}\\
$\mathbf{P}=-\alpha_m \mathbf{M} \, \ddot{\underline{\psi}}_n-\mathbf{K} ((1-\alpha_f)\underline{\psi}_{\textup{pred}}+\alpha_f \underline{\psi}_n)-\mathbf{C}( (1-\alpha_f)\dot{\underline{\psi}}_{\textup{pred}}+\alpha_f \underline{\dot{\psi}}_n)$ \vspace{1mm} \\
$\kappa=0$\vspace{2mm}\\
$\underline{\psi}_{n+1}^{\kappa}=\underline{\psi}_{\textup{pred}}$\vspace{2mm}\\
$\dot{\underline{\psi}}_{n+1}^{\kappa}=\underline{\dot{\psi}}_{\textup{pred}}$\vspace{2mm}\\
\While{$\displaystyle \frac{\|\ddot{\underline{\psi}}^{\kappa+1}_{n+1}-\ddot{\underline{\psi}}^\kappa_{n+1}\|}{\|\ddot{\underline{\psi}}^{\kappa+1}_{n+1}\|} > \textup{TOL}$ \vspace{1mm}}{
 \text{solve the algebraic system of equations} \vspace{2mm} \\
 $\mathbf{\overline{M}} \, \ddot{\underline{\psi}}^{\kappa+1}_{n+1} 
 	=\, \mathbf{P}-\bm{\mathcal{T}}(\underline{\dot{\psi}}_{n+1-\alpha_f}^\kappa)\underline{\psi}_{n+1-\alpha_f}^\kappa+\bm{\mathcal{N}}(\underline{\dot{\psi}}_{n+1-\alpha_f}^\kappa)\underline{\psi}_{n+1-\alpha_f}^\kappa
 $\vspace{1mm} \\
\hspace{-0mm}\text{perform corrector step} \vspace{2mm}  \\
$\underline{\psi}_{n+1}^{\kappa+1}=\underline{\psi}_{\textup{pred}}+\beta \Delta t^2 \ddot{\underline{\psi}}^{\kappa+1}_{n+1}$\vspace{2mm}\\
$\dot{\underline{\psi}}^{\kappa+1}_{n+1}=\dot{\underline{\psi}}_{\textup{pred}}+\gamma \Delta t \ddot{\underline{\psi}}^{\kappa+1}_{n+1}$\vspace{2mm}\\
$\kappa \mapsto \kappa+1$\\} 
$n \mapsto n+1$, $\, t \mapsto t+\Delta t$} 
\end{algorithm}
\medskip

Note that the effective mass matrix does not change with time in our algorithm. Therefore, we can compute its LU decomposition only in the first time step and use it when solving the algebraic system in the following time steps. \\
\indent After obtaining the acoustic velocity potential $\psi$, the acoustic pressure is then computed in a post-processing step as $u = \varrho \psi_t$.
\subsection{Neumann boundary conditions}
We also briefly comment on how the case of having inhomogeneous Neumann data would be numerically treated. In the presence of the Neumann boundary conditions
$$ \frac{\partial \psi}{\partial n}=g \quad \text{on} \quad \Gamma_n,$$ the equation \eqref{matrix_equation} changes into
\begin{equation*}
\begin{aligned}
\begin{cases}
\mathbf{M} \ddot{\underline{\psi}}+\mathbf{K}\underline{\psi}+\mathbf{C} \dot{\underline{\psi}}+ \bm{\mathcal{T}}(\dot{\underline{\psi}})\, \underline{\psi} - \bm{\mathcal{N}}(\dot{\underline{\psi}})\, \underline{\psi}=\bm{F}+\bm{\mathcal{B}}(\dot{\underline{\psi}}),\\
\underline{\psi}(0)=\underline{\psi}_0, \\
\dot{\underline{\psi}}(0)=\underline{\psi}_1.
\end{cases}
\end{aligned}
\end{equation*}
The terms on the right-hand side are given by 
\begin{equation*}
\begin{aligned}
& \bm{F}= \bigwedge\limits_{e=1}^{n_e} [F_{pq}], \quad F_{pq}=  \, \int_{\Gamma_n^e} (c^2g+b\dot{g}) N_p \,  \textup{d}x,\\
& \bm{\mathcal{B}}(\dot{\underline{\psi}})= \bigwedge\limits_{e=1}^{n_e} [B_{pq}], \quad B_{pq}=  c^2 k \, \int_{\Gamma_n^e}  g N_p \,  \Bigl( \sum_{r=1}^{n_\textup{en}} N_r \dot{\psi}_r \Bigr) \,  \textup{d}x.
\end{aligned}
\end{equation*}
The boundary tensor $\bm{\mathcal{B}}(\dot{\underline{\psi}})$ is a new term compared to the Kuznetsov equation. It originates from the term $\psi_t \Delta \psi$ in the equation. \\
\indent Considerations of how to avoid spurious reflections from the computational boundary in numerical simulations are beyond the scope of the present work. We refer the reader interested in this topic to the existing literature on absorbing boundary conditions \cite{engquist1977absorbing, shevchenko2015absorbing} and the perfectly matched layer (PML) technique \cite{berenger1994perfectly, manfred}. 
%%%%%%%%%%%%%%%%%%%%%%%%%%%%%%%% Numerical experiments %%%%%%%%%%%%%%%%%%%%%%%%%%%%%%%%%%%%
\section{Numerical experiments} \label{section:numerical_experiments}
We illustrate the behavior of the Blackstock equation through several numerical experiments in $1$D and $2$D settings for different parameter values.\\ 
\indent We set the Newmark parameters to $(\beta, \gamma)=(0.45, 0.75)$ and the Generalized-$\alpha$ parameters to $(\alpha_m, \alpha_f)=(1/2, 1/3)$, since they provided good results in simulations of the Westervelt equation, see \cite{muhr2017isogeometric}. The tolerance is set to $\textup{TOL}=10^{-8}$ in all experiments. All the numerical results have been obtained with the help of the GeoPDEs package~\cite{vazquez2016new} in MATLAB.
\subsection{Experiments in 1D} We first test the model in a $1$D channel example. As our medium, we use water with
\begin{align} \label{water_coefficients}
    \begin{cases}
    \text{the speed of sound} \ c=1500 \ \textup{m}/\textup{s},\\
    \text{the parameter of nonlinearity} \ B/A=5, \\
    \text{the sound diffusivity} \ b=6 \cdot 10^{-9} \ \textup{m}^2/\textup{s}, \ \text{and} \\
    \text{the density} \ \varrho=1000 \ \textup{kg}/\textup{m}^3;
    \end{cases}
\end{align}
 see~\cite[Ch. 5]{manfred}. %Recall that condition \eqref{degeneracy_condition} has to be fulfilled to avoid degeneracy of the equation, which in this water setting corresponds to the upper pressure bound of around $450 \ \textup{MPa}$. In practice this bound is quite high and allows for sufficiently large data to be taken. 
 We set the channel length to $\textup{l}=0.4 \ \textup{m}$ and the initial data to 
\begin{equation} \label{initial_data_experiments}
\begin{aligned}
    & \psi_0=0, \\
    & \psi_1= \mathcal{A} \cdot \textup{exp}\Bigl (-\frac{(x-\mu)^2}{2\sigma^2}\Bigr),
\end{aligned}
\end{equation}
with $\mathcal{A}=3\cdot 10^5 \,\textup{m}^2/\textup{s}^2$, $\mu=0.2$, $\sigma^2=0.01$. We first use cubic B-Splines as basis functions with the maximum $C^2$ global regularity. We take $801$ degrees of freedom in space and $801$ degrees of freedom in time with final time set to $T=0.1 \ \textup{ms}$. Figure~\ref{fig:channel_initial_data_snapshots} shows the evolution of the pressure wave over time. Average number of iterations per time step needed to converge to a solution was around $6$.
\begin{figure}[H]
\begin{center}
\input{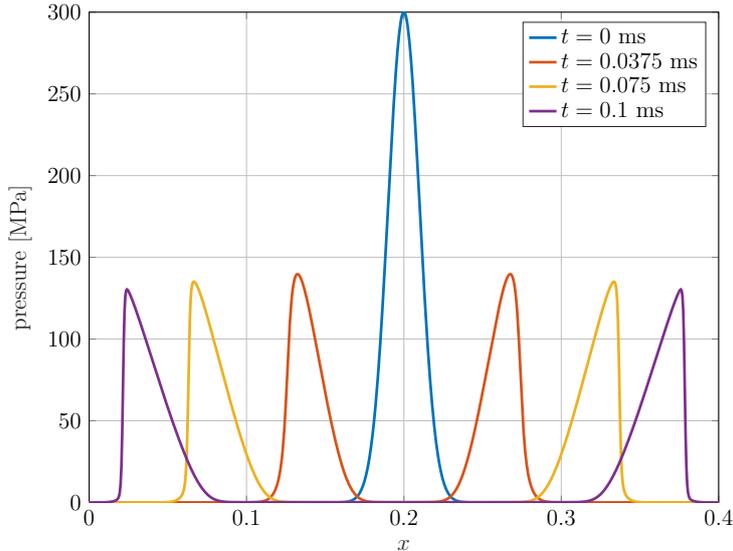}
\end{center}
\caption{Evolution of the pressure wave at different times \label{fig:channel_initial_data_snapshots}}
\end{figure}
Next we want to compare the performance of linear and quadratic splines as basis functions. We take the pressure computed with cubic splines on the fine grid as the reference solution. In Figure~\ref{fig:channel_initial_data_different_order}, we see the pressure profiles at final time when employing linear splines and quadratic splines with global $C^1$ regularity on a coarser grid which has $501$ degrees of freedom. We observe that quadratic splines perform much better, resulting in a wave that almost matches the reference solution. The error of the computed solutions is higher in the region where the wave front is steep. 
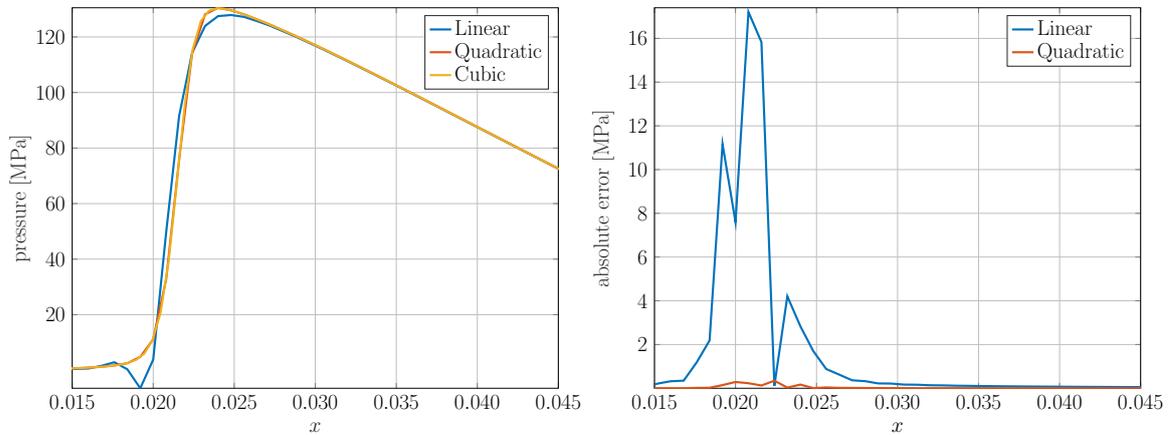
\begin{figure}[H]
\begin{center}
\begin{minipage}{.45\textwidth}
\definecolor{mycolor1}{rgb}{0.00000,0.44700,0.74100}%
\definecolor{mycolor2}{rgb}{0.85000,0.32500,0.09800}%
\definecolor{mycolor3}{rgb}{0.92900,0.69400,0.12500}%
\begin{tikzpicture}[scale=0.47][font=\huge]

\begin{axis}[%
width=5.358in,
height=4.226in,
scaled ticks=false,
at={(0.899in,0.57in)},
scale only axis,
x tick label style={
        /pgf/number format/.cd,
            fixed,
            fixed zerofill,
            precision=3,
        /tikz/.cd
            },
xmin=0.015,
xmax=0.045, 
xtick={0.015, 0.020, 0.025, 0.030, 0.035, 0.040, 0.045},
xlabel style={font=\color{white!15!black}},
xlabel={$x$},
ymin=-6.49574941278012,
ymax=130.506413357092, ytick={20, 40, 60, 80, 100, 120},
ylabel style={at={(axis description cs:-.07,.5)}, font=\color{white!15!black}},
ylabel={pressure [MPa]},
axis background/.style={fill=white},
xmajorgrids,
ymajorgrids,
legend style={legend cell align=left, align=left, draw=white!15!black}
]
\addplot [color=mycolor1][line width=0.6mm]
  table[row sep=crcr]{%
0.0143999999999949	0.487011351421117\\
0.015199999999993	0.487700712975979\\
0.0160000000000053	0.603470897074715\\
0.0168000000000035	1.57636172543083\\
0.0176000000000016	2.86784372372048\\
0.0183999999999997	0.350276015443441\\
0.0191999999999979	-6.49574941278013\\
0.019999999999996	3.84835035065767\\
0.0207999999999942	49.7203926547292\\
0.0216000000000065	91.6088636796557\\
0.0224000000000046	114.227302853108\\
0.0232000000000028	123.955524726647\\
0.0240000000000009	127.497191269787\\
0.024799999999999	127.895758836778\\
0.0255999999999972	127.229297294708\\
0.0263999999999953	125.719608629682\\
0.0271999999999935	124.064274461188\\
0.0288000000000039	120.089732863975\\
0.0304000000000002	115.775220214196\\
0.0319999999999965	111.268227163621\\
0.033600000000007	106.636565961491\\
0.0352000000000032	101.921307267713\\
0.0367999999999995	97.1514644319994\\
0.0391999999999939	89.9419982851165\\
0.0431999999999988	77.9174157931987\\
0.0447999999999951	73.1436647274256\\
0.0455999999999932	70.770419929989\\
};
\addlegendentry{Linear}

\addplot [color=mycolor2][line width=0.6mm]
  table[row sep=crcr]{%
0.0143999999999949	0.534524848001212\\
0.015199999999993	0.704178503379723\\
0.0159999999999911	0.92389391734983\\
0.0167999999999893	1.22981607611035\\
0.0175999999999874	1.70631355112459\\
0.018400000000014	2.51342159394559\\
0.0192000000000121	4.81555404068496\\
0.0200000000000102	11.1574495341941\\
0.0208000000000084	32.7446818212571\\
0.0216000000000065	75.9037623643755\\
0.0224000000000046	113.996266129503\\
0.0232000000000028	128.132993407869\\
0.0240000000000009	130.506413357092\\
0.024799999999999	129.589275740855\\
0.0255999999999972	128.147340450225\\
0.0263999999999953	126.324324035316\\
0.0271999999999935	124.44548775675\\
0.0279999999999916	122.398373641408\\
0.0287999999999897	120.320617553827\\
0.0295999999999879	118.147029583981\\
0.030399999999986	115.947113940781\\
0.0312000000000126	113.688733129164\\
0.0320000000000107	111.407951308153\\
0.033600000000007	106.754761597989\\
0.0352000000000032	102.023557234317\\
0.0367999999999995	97.2412796616834\\
0.039999999999992	87.6046092448535\\
0.043200000000013	77.9760229318128\\
0.0448000000000093	73.1970568465725\\
0.0456000000000074	70.8214225494455\\
};
\addlegendentry{Quadratic}

\addplot [color=mycolor3][line width=0.6mm]
  table[row sep=crcr]{%
0.0149625935162021	0.648030190294463\\
0.0154613466334297	0.768038248466922\\
0.0159600997506288	0.911676094938883\\
0.016458852867828	1.08598997839232\\
0.0169576059850272	1.30333574618757\\
0.0174563591022547	1.58734415507524\\
0.0179551122194539	1.98897561203816\\
0.0184538653366531	2.6227263494853\\
0.0189526184538522	3.75625395611871\\
0.0194513715710798	6.00863477826226\\
0.0199501246882789	10.7451103721102\\
0.0204488778054781	20.5885870251399\\
0.0209476309227057	38.9963385705131\\
0.0214463840399048	66.5176977791942\\
0.021945137157104	95.4030275355022\\
0.0224438902743032	115.684980665465\\
0.0229426433915307	125.71677954902\\
0.0234413965087299	129.435299567922\\
0.0239401496259291	130.327311188873\\
0.0244389027431282	130.083770727403\\
0.0249376558603558	129.372367982506\\
0.0254364089775549	128.442361960055\\
0.0259351620947541	127.392698541581\\
0.0264339152119817	126.265355469742\\
0.0269326683291808	125.08172062237\\
0.02743142144638	123.853623396762\\
0.0279301745635792	122.589040046175\\
0.0284289276808067	121.293477721313\\
0.0289276807980059	119.971277822816\\
0.0294264339152051	118.625822733911\\
0.0304239401496318	115.876050037935\\
0.0314214463840301	113.062016115842\\
0.0324189526184568	110.197146340596\\
0.0334164588528552	107.291979341685\\
0.0349127182044811	102.876943782624\\
0.0364089775561069	98.4132301521321\\
0.038403990024932	92.4170184259281\\
0.0433915211970088	77.4022897533641\\
0.0453865336658339	71.4542937721525\\
};
\addlegendentry{Cubic}

\end{axis}
\end{tikzpicture}%
\end{minipage}
\hspace{5mm}
\begin{minipage}{.45\textwidth}
\definecolor{mycolor1}{rgb}{0.00000,0.44700,0.74100}%
\definecolor{mycolor2}{rgb}{0.85000,0.32500,0.09800}%
\begin{tikzpicture}[scale=0.47][font=\huge]

\begin{axis}[%
width=5.358in,
height=4.226in,
scaled ticks=false,
at={(0.899in,0.57in)},
x tick label style={
        /pgf/number format/.cd,
            fixed,
            fixed zerofill,
            precision=3,
        /tikz/.cd
            },
scale only axis,
xmin=0.015,
xmax=0.045, 
xtick={0.015, 0.020, 0.025, 0.030, 0.035, 0.040, 0.045},
xlabel={$x$},
ymin=0,
ymax=17.4, ytick={2, 4, 6, 8, 10, 12, 14, 16},
ylabel style={font=\color{white!15!black}},
ylabel={absolute error [MPa]},
axis background/.style={fill=white},
xmajorgrids,
ymajorgrids,
legend style={legend cell align=left, align=left, draw=white!15!black}
]
\addplot [color=mycolor1][line width=0.6mm]
  table[row sep=crcr]{%
0.0143999999999984	0.0482260925965612\\
0.0152000000000001	0.214837710234139\\
0.0160000000000018	0.320901357834408\\
0.0167999999999999	0.347482345225337\\
0.0176000000000016	1.18036857175741\\
0.0183999999999997	2.18655978909697\\
0.0192000000000014	11.1653971629357\\
0.0199999999999996	7.59675538558199\\
0.0208000000000013	17.206048109746\\
0.0215999999999994	15.8319085483117\\
0.0224000000000011	0.114918458654984\\
0.0231999999999992	4.20707362203471\\
0.0240000000000009	2.84075061982684\\
0.024799999999999	1.70146665072831\\
0.0256000000000007	0.879163188911321\\
0.0263999999999989	0.624375319075181\\
0.0272000000000006	0.364159530454533\\
0.0279999999999987	0.324326656217739\\
0.0288000000000004	0.222363804243297\\
0.0295999999999985	0.214979203606262\\
0.0304000000000002	0.167638376884131\\
0.0311999999999983	0.162019647225424\\
0.032	0.137627491210356\\
0.0328000000000017	0.131007437665314\\
0.0335999999999999	0.117176130828888\\
0.0344000000000015	0.110514004843235\\
0.0351999999999997	0.101760348329574\\
0.0360000000000014	0.0957634607109128\\
0.0367999999999995	0.0895821960600465\\
0.0376000000000012	0.0844555820910493\\
0.0383999999999993	0.0796929743888839\\
0.039200000000001	0.0753791545115696\\
0.0399999999999991	0.0714900993304717\\
0.0408000000000008	0.067848105915786\\
0.041599999999999	0.0645571261318771\\
0.0424000000000007	0.0614444250781396\\
0.0431999999999988	0.0585956329192356\\
0.0440000000000005	0.0558939651086909\\
0.0447999999999986	0.0533866872503026\\
0.0456000000000003	0.0510039386759402\\
};
\addlegendentry{Linear}

\addplot [color=mycolor2][line width=0.6mm]
  table[row sep=crcr]{%
0.0144	0.000712596016458122\\
0.0152	0.00164008016959694\\
0.016	0.000478337559285003\\
0.0168	0.000936695904864004\\
0.0176000000000001	0.0188383991615237\\
0.0184	0.0234142105948073\\
0.0192000000000001	0.145906290529416\\
0.0200000000000001	0.287656202045547\\
0.0208	0.230337276273932\\
0.0216000000000001	0.126807233031511\\
0.0224	0.345955182260111\\
0.0232000000000001	0.0296049408127069\\
0.0240000000000001	0.168471467477292\\
0.0248	0.00794974665154519\\
0.0256000000000001	0.0388799666059613\\
0.0264	0.0196599134410024\\
0.0272000000000001	0.0170537651075274\\
0.0280000000000001	0.0110462510077357\\
0.0288	0.00852088560874759\\
0.0296000000000001	0.00565690545138714\\
0.0304000000000001	0.00425534970110653\\
0.0312000000000001	0.00280928111468254\\
0.0320000000000001	0.0020966533215046\\
0.0328000000000001	0.0013592488564253\\
0.0336000000000001	0.00101950566877423\\
0.0344000000000001	0.000642877298623334\\
0.0352000000000001	0.000489618274569503\\
0.0360000000000001	0.000297056454092282\\
0.0368000000000001	0.000233033624023216\\
0.0376000000000001	0.000133863309547289\\
0.0384000000000001	0.000110239302307358\\
0.0392000000000001	5.86299412548286e-05\\
0.0400000000000001	5.1981622979036e-05\\
0.0408000000000001	2.48064738362919e-05\\
0.0416000000000001	2.44647731930114e-05\\
0.0424000000000001	1.00512096732897e-05\\
0.0432000000000001	1.15056948959968e-05\\
0.0440000000000001	3.84468708930497e-06\\
0.0456000000000001	1.31921951473934e-06\\
};
\addlegendentry{Quadratic}

\end{axis}
\end{tikzpicture}%
\end{minipage}
\end{center}
\caption{\textbf{Left:} Comparison of the solutions at final time $t=0.1$ ms computed by using linear and quadratic splines against the reference solution; \textbf{Right:} Error of the solutions obtained by using linear and quadratic splines} \label{fig:channel_initial_data_different_order}
\end{figure}
\indent Since quadratic splines result in a reasonably accurate solution, we continue with employing quadratic basis functions with maximum $C^1$ global regularity in the following experiments. \\
\indent Next we want to see how the damping parameter $b$ and the parameter of nonlinearity $B/A$ influence the behavior of the model. A similar study was provided in \cite{hoffelner2002simulation} for the Kuznetsov equation with Neumann data. We also discuss how the size of the initial data affects the pressure wave and plot the energy decay with respect to time. \\
\subsubsection{Influence of the strong damping} First, we set all the parameters as before in \eqref{water_coefficients}, but vary the sound diffusivity $$b \in \{0, 0.1, 1, 10, 100\}\,\textup{m}^2/\textup{s}.$$ %\, \textup{m}^2/\textup{s}.$$
Figure~\ref{fig:channel_initial_data_different_b} displays how $b$ influences the wave behavior. The snapshots were taken at $t=0.8 \, \textup{ms}$. Note that when $b \approx 0 \, \textup{m}^2/\textup{s}$ and the amplitude of the initial data is sufficiently large, a steepening of the wave front can be observed. Eventually this nonlinear steepening will develop into a vertical wave front; see also Figure~\ref{fig:channel_initial_data_snapshots}. As $b$ increases, the amplitude of the wave is damped and the nonlinear behavior gets less and less pronounced. For larger values of $b$, we can see that steepening of the wave ceases, meaning that there is no formation of higher harmonics. 
\begin{figure}[H]
\input{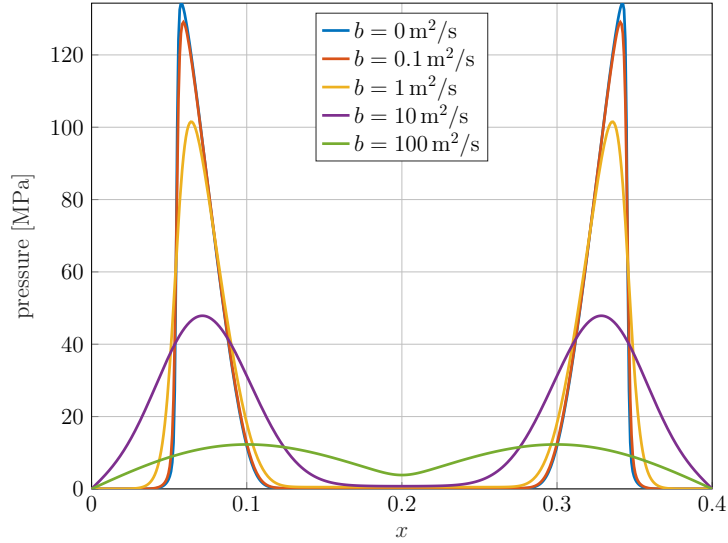}
\caption{Influence of the $b$-damping on the behavior of the model  \label{fig:channel_initial_data_different_b}}
\end{figure}
\subsubsection{Influence of the data size} Next we want to illustrate how the smallness of the initial data influences the behavior and regularity of the pressure wave. We take the same water setting as before with coefficients as in \eqref{water_coefficients}. We then plot snapshots of the ultrasound wave $t= 0.9 \, \textup{ms}$ developed from the initial wave given by \eqref{initial_data_experiments}, but we take different amplitudes for the Gaussian function $\psi_1$: $$\mathcal{A} \in \{1.5, 5, 8, 15\} \cdot 10^4 \,\textup{m}^2/\textup{s}^2.$$
Figure~\ref{fig:channel_initial_data_different_size} displays the influence that the size of the initial data has on the behavior of the wave. As seen in the previous experiment, for high-amplitude waves an obvious steepening of the wave front can be observed. When the initial data is sufficiently small, the behavior resembles the one of a linear wave with almost no distortion as the wave travels. 
\begin{figure}[H]
\input{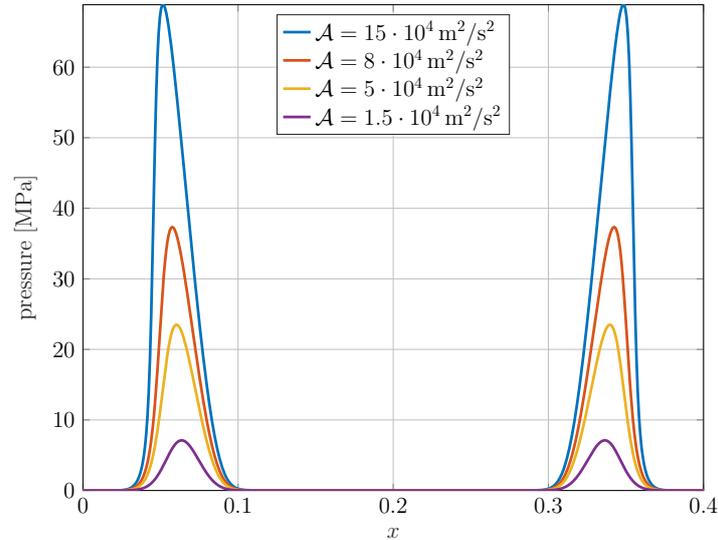}
\caption{Influence of the size of the initial data on the behavior of the model
\label{fig:channel_initial_data_different_size}}
\end{figure}
\subsubsection{Energy decay}
Figure~\ref{fig:channel_initial_data_energy_decay} shows how the energy decays with respect to time. The medium is again taken to be water with coefficients as in \eqref{water_coefficients}, but we vary the sound diffusivity $b$. The initial data is given as before by \eqref{initial_data_experiments} with the high amplitude $\mathcal{A}=3 \cdot 10^5 \,\textup{m}^2/\textup{s}^2$.  As observed in our theoretical considerations, a certain amount of dissipation is needed for the energy to decay with time. When a larger damping is present in the model, the results agree with the exponential decay rates of the energy obtained for the one-dimensional case in Theorem~\ref{theorem:1D}.
\begin{figure}[H]
\input{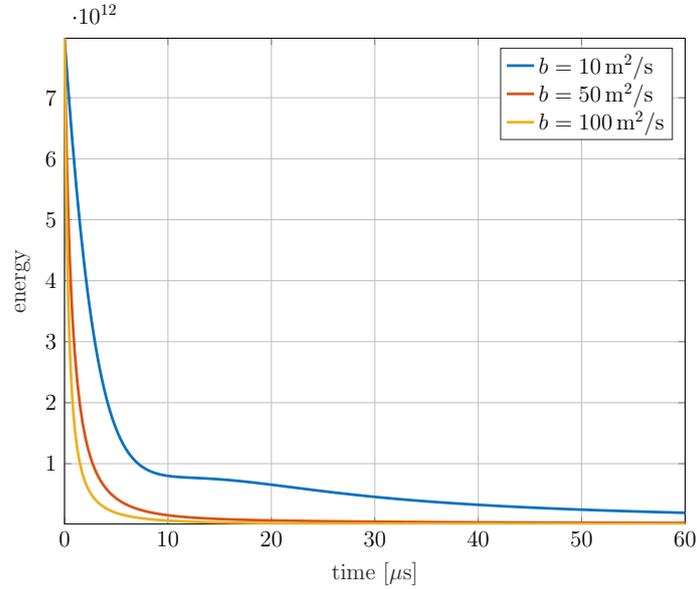}
\caption{Decay of the energy $|\Delta \psi(t)|^2_{L^2}+|\nabla \psi_t(t)|^2_{L^2}$ with respect to time with different dissipation levels \label{fig:channel_initial_data_energy_decay}}
\end{figure}
\subsubsection{Different models} Figure~\ref{fig:channel_comparison_Blackstock_Kuznetsov_Westervelt} displays differences in the pressure profiles obtained by employing the Blackstock, the Kuznetsov \eqref{Kuznetsov_eq}, and the Westervelt equation \eqref{Westervelt_eq} in the same setting of a $1$D channel, where the medium is taken to be water. 
\begin{figure}[H]
\input{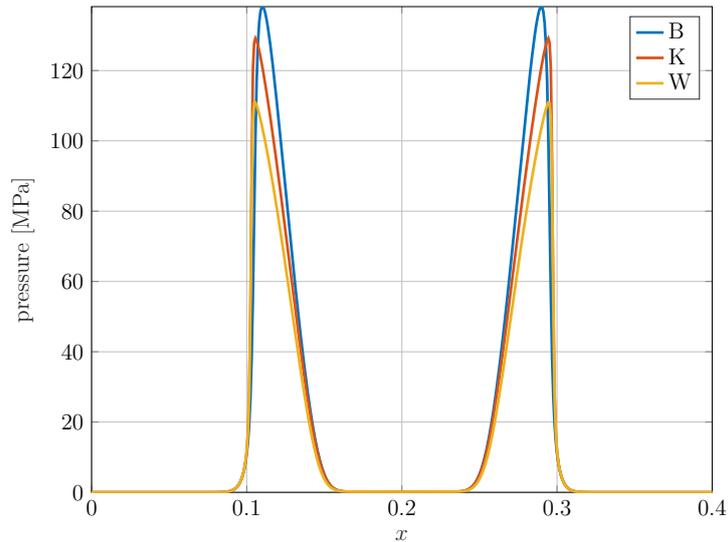}
\caption{Plot of the pressure waves obtained by using the Blackstock (B), the Kuznetsov (K), and the Westervelt (W) equations at $t=50 \, \mu \textup{s}$  \label{fig:channel_comparison_Blackstock_Kuznetsov_Westervelt}}
\end{figure}
\indent The coefficients are taken as in \eqref{water_coefficients} and the initial data as in \eqref{initial_data_experiments} with the amplitude $\mathcal{A}=3 \cdot 10^5 \,\textup{m}^2/\textup{s}^2$. We employed a fixed point iteration with respect to the second time derivative to solve the Kuznetsov and the Westervelt  equations. The average number of iterations per time step needed to converge to a solution was around $5$ for the Blackstock, $14$ for the Kuznetsov, and around $19$ for the Westervelt equation. \\
\indent A difference in the pressure profiles can easily be observed for the high-amplitude ultrasound waves. We refer also to the work in \cite{ChristovJordan}, where these models were compared to the Euler equations in the $1$D non-viscous case. The study showed that the Blackstock equation performed the best, followed by the inviscid Kuznetsov equation. The inviscid Westervelt equation had the poorest performance. 
\subsubsection{Neumann excitation.} Next we take the same water setting as before, but this time with Neumann excitation on the left side of the channel (and homogeneous Neumann conditions on the right end). For the source condition, we take a modulated sinusoidal wave 
 \begin{align} \label{Neumann_excitation}
 	g= \begin{cases} g_0\, \sin(\omega t)\,(1+ \sin(wt/4)) , \ t> 2 \pi/w, \smallskip \\ g_0\, \sin(\omega t), \hspace{25mm}  t \leq 2\pi/w. \end{cases}
 \end{align}
The amplitude of the source is taken to be $g_0=20 \pi \ \textup{m}/\textup{s}$ and the frequency $f=70 \ \textup{kHz}$. The angular frequency is then given by $\omega=2\pi f$. We take $451$ degrees of freedom in space and $1000$ time steps with final time set to $0.12 \ \textup{ms}$. \\
\indent We set $c=1500 \ \textup{m}/\textup{s}$, $b=6 \cdot 10^{-9} \ \textup{m}^2/\textup{s}$, $\varrho=1000 \ \textup{kg}/\textup{m}^3$ as before and now vary the parameter of nonlinearity of the medium. Typical values of this parameter in different media can be found, for instance, in \cite{dunn2015springer}. We take $$B/A \in \{0, 3, 7\}.$$
\begin{figure}[H]
\input{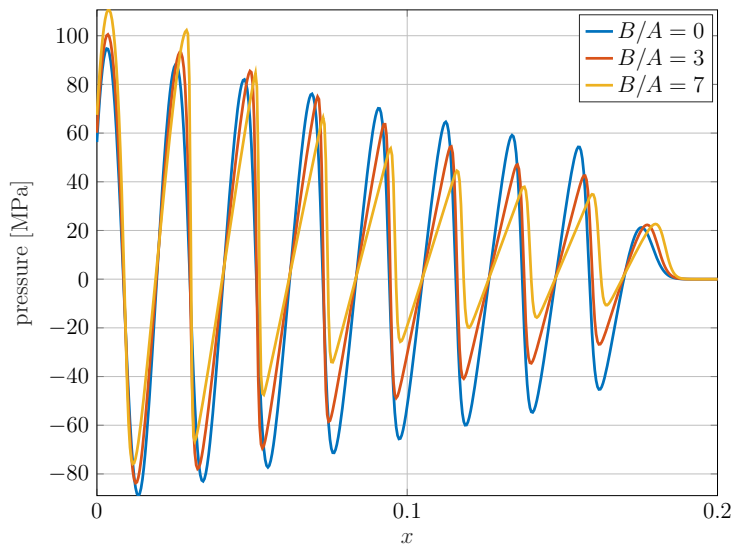}
\caption{Influence of the parameter of nonlinearity $B/A$ on the behavior of the model  \label{fig:channel_different_BA}}
\end{figure}
Figure~\ref{fig:channel_different_BA} illustrates how the wave is affected by the material nonlinearity. As expected, the larger the parameter of nonlinearity is, the closer to the source the steepening of the wave begins. If the data is sufficiently large and not much dissipation is present, eventually the steepening becomes vertical and a wave form with a sawtooth shape develops.
\subsection{High-intensity focused ultrasound}
Finally, we plot the acoustic pressure a two-dimen\-sional domain, corresponding to a high-intensity focused ultrasound (HIFU) setting. The excitation and focusing in HIFU applications is often achieved by an array of transducers placed on a spherical surface; see \cite[Ch. 12]{manfred}. Therefore, for our domain we take a rectangle $[0, 0.08 \, \textup{m}] \times [0, 0.12 \, \textup{m}]$ with a curved lower boundary belonging to a circle with center $(0.04 \, \textup{m}, 0.03 \, \textup{m})$ and radius $R=0.05 \, \textup{m}$.
We employ quadratic splines and take $282 \times 452$ degrees of freedom and $1500$ time steps with the final time set to $0.1 \, \textup{ms}$. The excitation on the curved boundary is again taken to be the modulated sine wave \eqref{Neumann_excitation} with amplitude $g= 20 \, \textup{m}/\textup{s}$ and frequency $f=100 \, \textup{kHz}$, whereas we set Neumann conditions to zero on the other boundaries. \\
\indent Figure \ref{fig:2D_wave} shows the propagation and self-focusing of the pressure wave. Note that due to the lack of the absorbing conditions, reflections can be observed off the sides of the computational domain. In Figure~\ref{fig:pressure_axis}, we can see a snapshot of the pressure changes along the axis of symmetry of the domain. 
%%%%%%%%%%%%%%%%%%%%%%%%%%%%%%%%%%%%
\begin{figure}[H]
\begin{center}
\hspace{-18mm}
\begin{minipage}{8cm}
	\includegraphics[trim = 13cm 0cm 0cm 2cm, clip, scale=0.24]{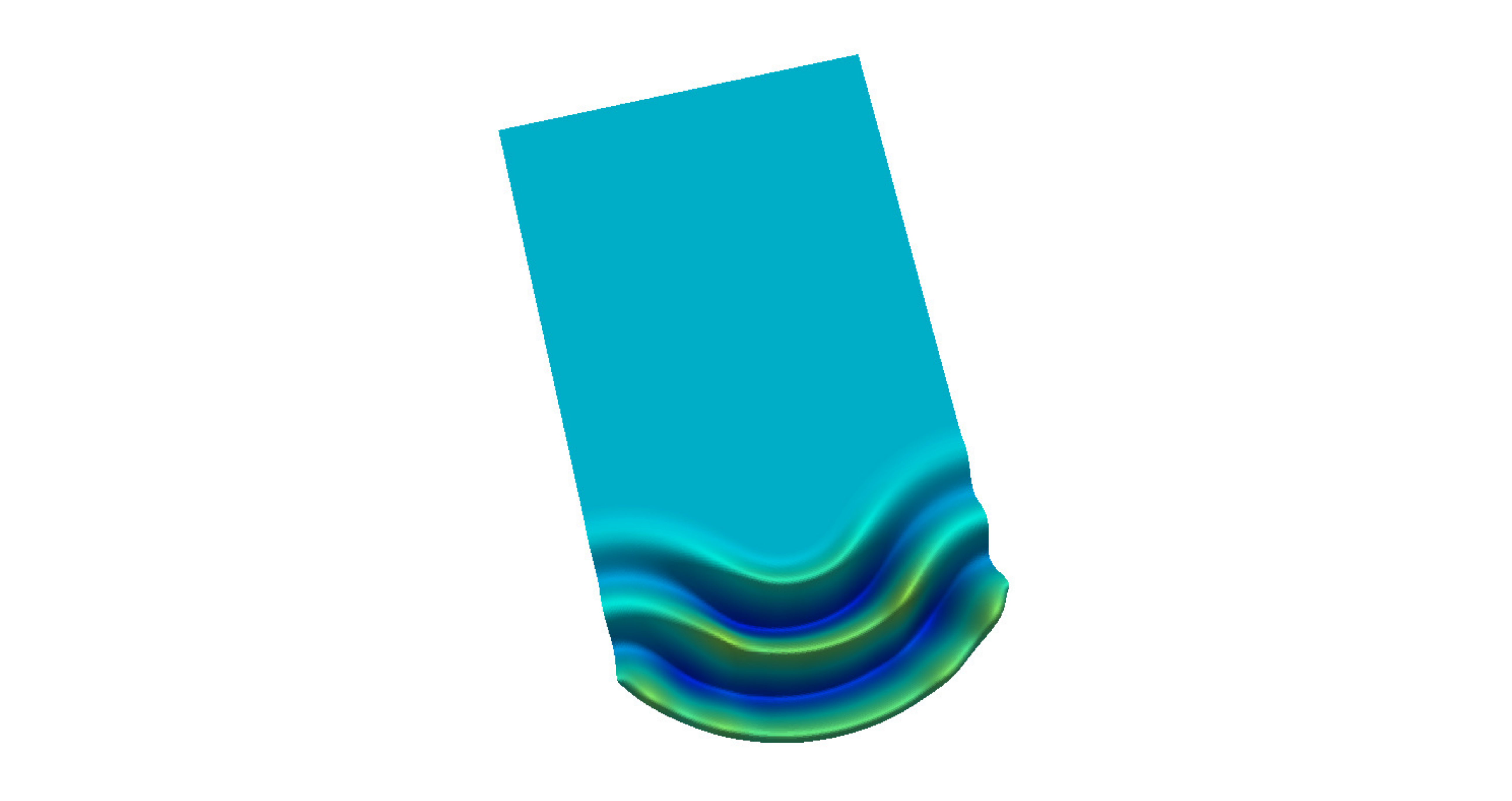}
\end{minipage}
\hspace{-18mm}
\begin{minipage}{6cm}
	\includegraphics[trim = 13cm 0cm 0cm 2cm, clip, scale=0.24]{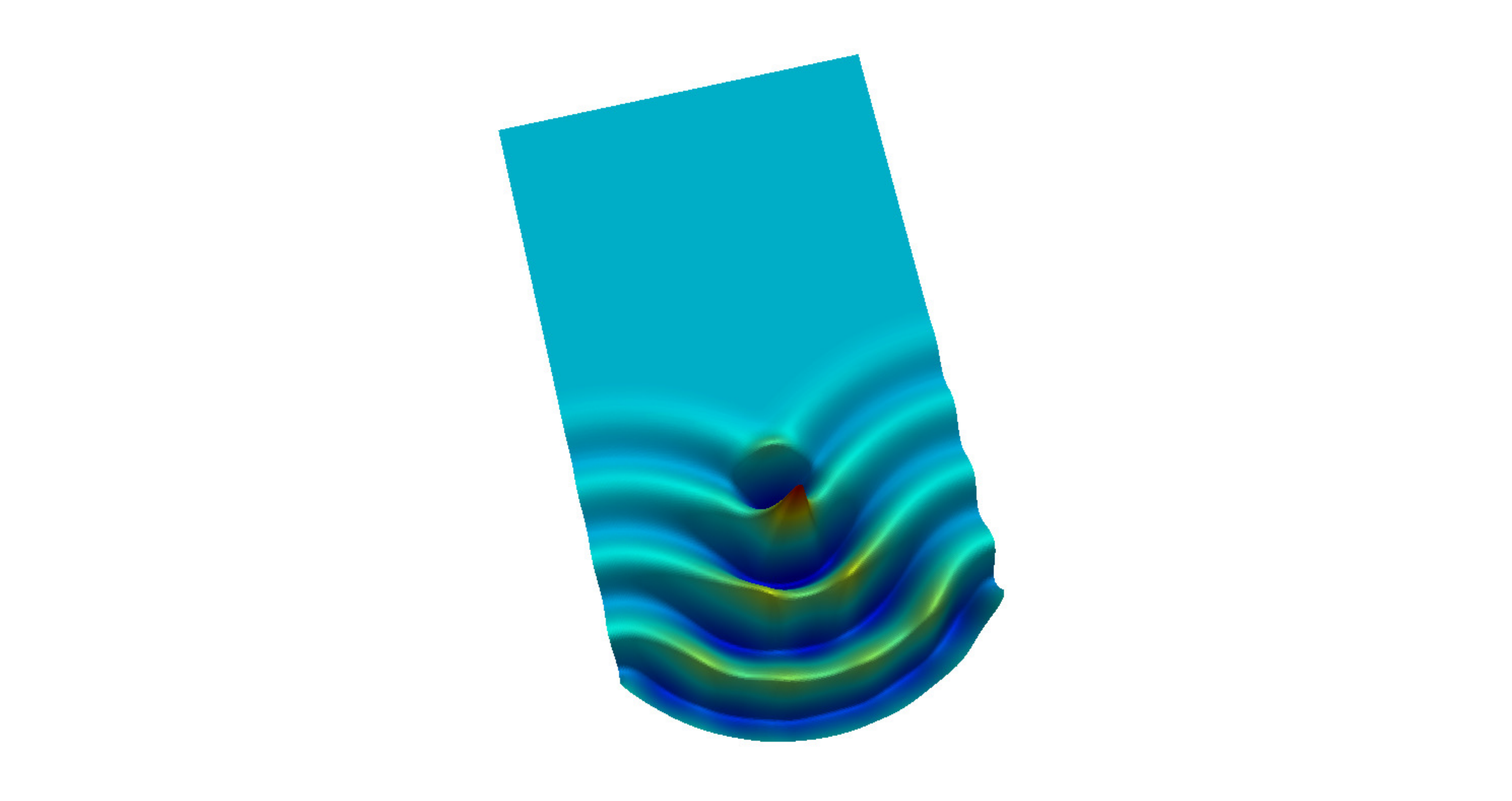}
\end{minipage}
\end{center}
\vspace{-5mm}
\begin{center}
\hspace{-18mm}
\begin{minipage}{8cm}
	\includegraphics[trim = 13cm 0cm 0cm 0cm, clip, scale=0.24]{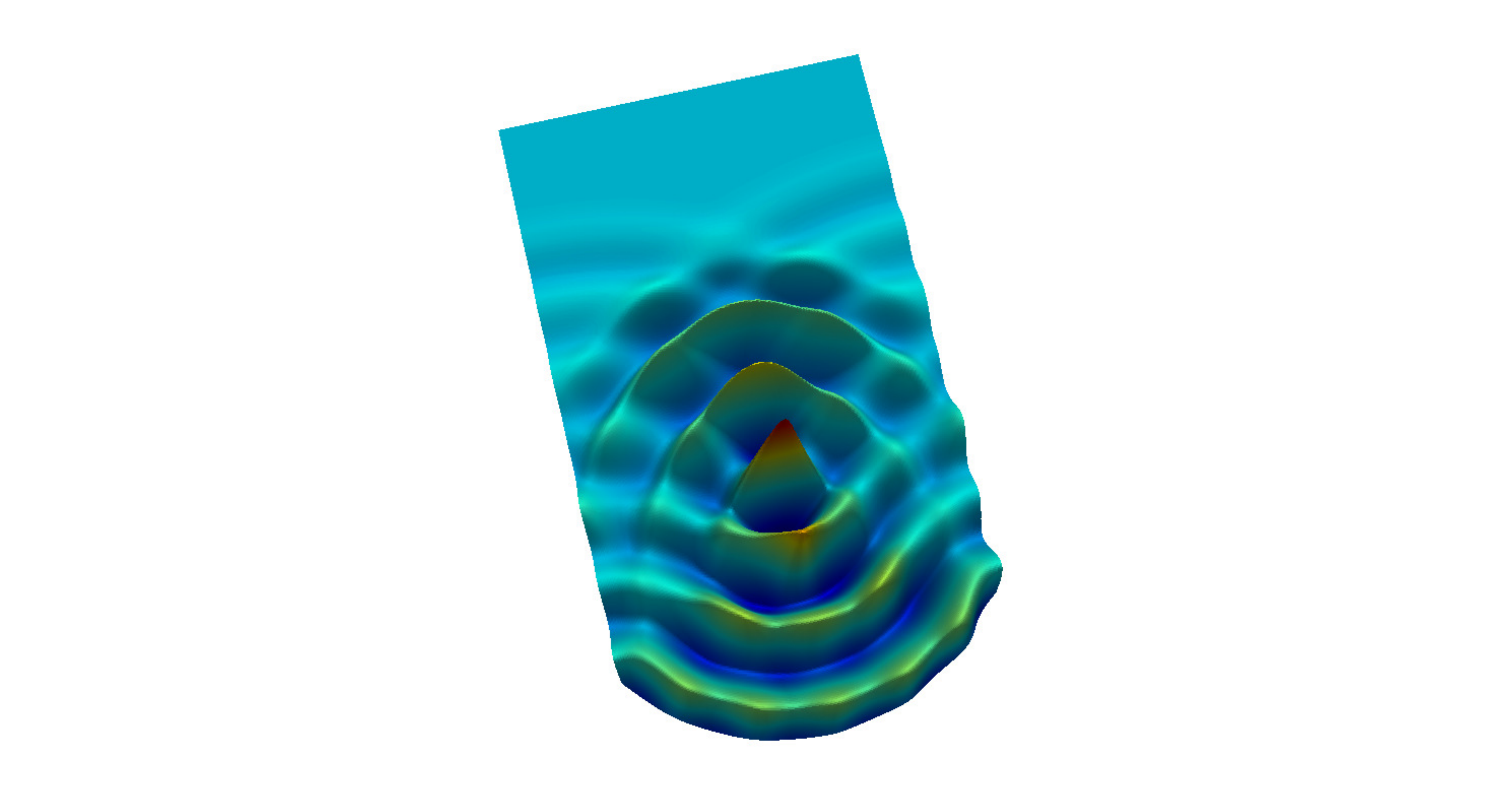}
\end{minipage}
\vspace{-5mm}
\hspace{-18mm}
\begin{minipage}{6cm}
\includegraphics[trim = 13cm 0cm 0cm 0cm, clip, scale=0.24]{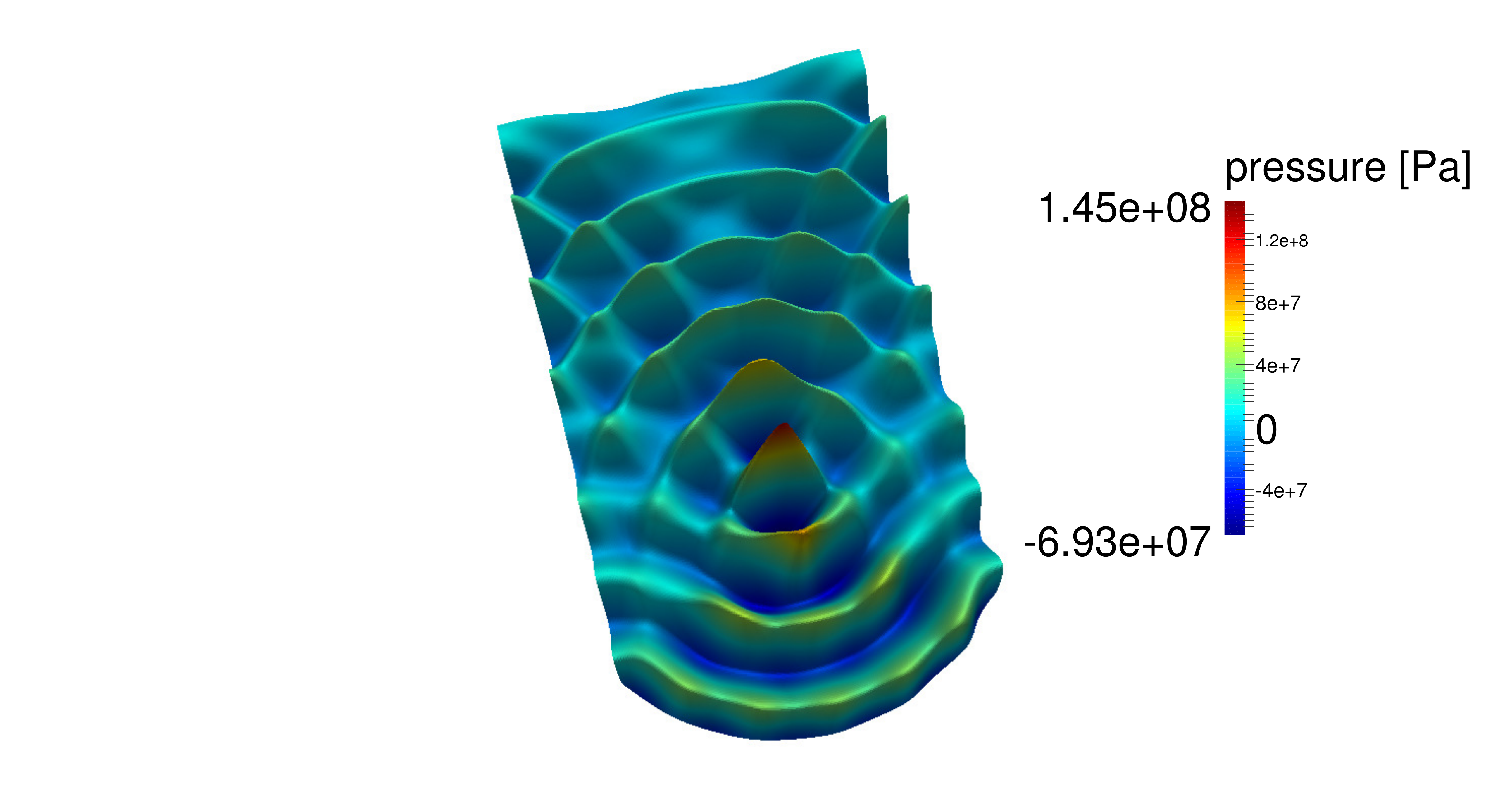}
\end{minipage}
\end{center}
\caption{Evolution of the pressure wave with focusing \label{fig:2D_wave}}
\end{figure}

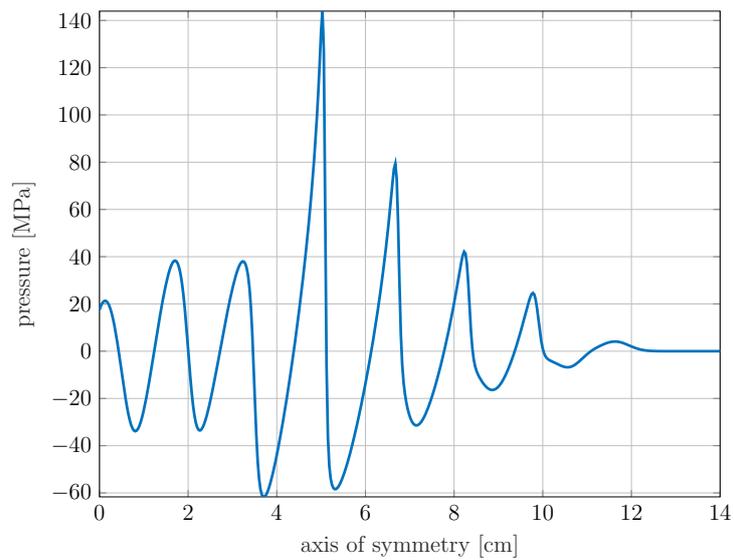
\begin{figure}[H]
\begin{center}
\definecolor{mycolor1}{rgb}{0.00000,0.44700,0.74100}%
\begin{tikzpicture}[scale=0.6][font=\LARGE]

\begin{axis}[%
width=5.358in,
height=4.226in,
at={(0.899in,0.57in)},
scale only axis,
xmin=0,
xmax=14, xtick={0, 2, 4, 6, 8, 10, 12, 14},
xlabel style={font=\color{white!15!black}},
xlabel={axis of symmetry [cm]},
ymin=-61.665,
ymax=143.93,
ylabel style={font=\color{white!15!black}},
ylabel={pressure [MPa]},
axis background/.style={fill=white},
xmajorgrids,
ymajorgrids,
legend style={legend cell align=left, align=left, draw=white!15!black}
]
\addplot [color=mycolor1][line width=0.6mm]
  table[row sep=crcr]{%
0	17.486\\
0.0310421286031044	19.11\\
0.0620842572062088	20.307\\
0.0931263858093132	21.057\\
0.124168514412418	21.346\\
0.155210643015522	21.165\\
0.186252771618626	20.509\\
0.217294900221731	19.38\\
0.248337028824835	17.784\\
0.27937915742794	15.735\\
0.310421286031044	13.252\\
0.341463414634148	10.366\\
0.403547671840357	3.55170000000001\\
0.465631929046566	-4.25790000000001\\
0.589800443458984	-20.154\\
0.620842572062088	-23.601\\
0.651884700665192	-26.642\\
0.682926829268297	-29.205\\
0.713968957871401	-31.231\\
0.745011086474506	-32.689\\
0.77605321507761	-33.565\\
0.807095343680714	-33.866\\
0.838137472283819	-33.612\\
0.869179600886923	-32.839\\
0.900221729490028	-31.588\\
0.931263858093132	-29.904\\
0.962305986696236	-27.835\\
0.993348115299341	-25.429\\
1.05543237250555	-19.79\\
1.11751662971176	-13.329\\
1.21064301552107	-2.7559\\
1.36585365853659	15.1\\
1.4279379157428	21.652\\
1.49002217294901	27.517\\
1.55210643015522	32.435\\
1.58314855875832	34.444\\
1.61419068736143	36.095\\
1.64523281596453	37.338\\
1.67627494456764	38.114\\
1.70731707317074	38.358\\
1.73835920177385	37.993\\
1.76940133037695	36.933\\
1.80044345898006	35.082\\
1.83148558758316	32.34\\
1.86252771618626	28.613\\
1.89356984478937	23.838\\
1.92461197339247	18.024\\
1.98669623059868	3.9306\\
2.04878048780489	-10.917\\
2.07982261640799	-17.441\\
2.1108647450111	-22.898\\
2.1419068736142	-27.155\\
2.17294900221731	-30.225\\
2.20399113082041	-32.218\\
2.23503325942352	-33.277\\
2.26607538802662	-33.553\\
2.29711751662973	-33.178\\
2.32815964523283	-32.267\\
2.35920177383593	-30.913\\
2.39024390243904	-29.195\\
2.42128603104211	-27.175\\
2.48337028824832	-22.426\\
2.54545454545453	-16.987\\
2.63858093126385	-8.0307\\
2.88691796008868	16.702\\
2.94900221729489	22.325\\
3.0110864745011	27.434\\
3.07317073170731	31.851\\
3.10421286031041	33.73\\
3.13525498891352	35.342\\
3.16629711751662	36.638\\
3.19733924611972	37.557\\
3.22838137472283	38.014\\
3.25942350332593	37.891\\
3.29046563192904	37.021\\
3.32150776053214	35.155\\
3.35254988913525	31.926\\
3.38359201773835	26.793\\
3.41463414634146	19.054\\
3.44567627494456	8.0821\\
3.50776053215077	-21.494\\
3.53880266075387	-35.766\\
3.56984478935698	-46.754\\
3.60088691796008	-54.084\\
3.63192904656319	-58.46\\
3.66297117516629	-60.76\\
3.6940133037694	-61.665\\
3.7250554323725	-61.617\\
3.7560975609756	-60.895\\
3.78713968957871	-59.678\\
3.81818181818181	-58.078\\
3.84922394678492	-56.174\\
3.91130820399113	-51.654\\
3.97339246119734	-46.398\\
4.03547671840354	-40.562\\
4.12860310421286	-30.913\\
4.22172949002217	-20.33\\
4.31485587583148	-8.87020000000001\\
4.4079822616408	3.4913\\
4.50110864745011	16.862\\
4.59423503325942	31.471\\
4.65631929046563	42.097\\
4.71840354767184	53.671\\
4.78048780487805	66.537\\
4.84257206208426	81.29\\
4.90465631929047	99.022\\
4.93569844789357	109.69\\
4.96674057649668	121.68\\
4.99778270509978	135.2\\
5.02882483370288	143.93\\
5.05986696230599	127.64\\
5.1219512195122	-6.02019999999999\\
5.1529933481153	-36.264\\
5.18403547671841	-48.603\\
5.21507760532151	-53.918\\
5.24611973392462	-56.618\\
5.27716186252772	-57.929\\
5.30820399113082	-58.41\\
5.33924611973393	-58.316\\
5.37028824833703	-57.8\\
5.40133037694014	-56.958\\
5.43237250554324	-55.854\\
5.46341463414635	-54.535\\
5.52549889135256	-51.374\\
5.58758314855876	-47.655\\
5.64966740576497	-43.488\\
5.71175166297118	-38.94\\
5.80487804878049	-31.495\\
5.89800443458981	-23.368\\
5.99113082039912	-14.579\\
6.08425720620843	-5.09800000000001\\
6.17738359201775	5.15799999999999\\
6.27050997782706	16.346\\
6.33259423503327	24.449\\
6.39467849223948	33.206\\
6.45676274944569	42.793\\
6.5188470066519	53.466\\
6.5809312638581	65.44\\
6.61197339246121	71.687\\
6.64301552106431	77.176\\
6.67405764966742	79.338\\
6.70509977827052	73.178\\
6.73614190687363	54.705\\
6.76718403547673	29.156\\
6.79822616407984	7.61060000000001\\
6.82926829268294	-5.65020000000001\\
6.86031042128604	-13.279\\
6.89135254988915	-18.166\\
6.92239467849225	-21.728\\
6.95343680709536	-24.515\\
6.98447893569846	-26.722\\
7.01552106430157	-28.442\\
7.04656319290467	-29.728\\
7.07760532150778	-30.625\\
7.10864745011088	-31.171\\
7.13968957871398	-31.401\\
7.17073170731709	-31.349\\
7.20177383592016	-31.042\\
7.2328159645233	-30.508\\
7.26385809312637	-29.772\\
7.29490022172951	-28.852\\
7.32594235033258	-27.77\\
7.35698447893569	-26.539\\
7.4190687361419	-23.688\\
7.4811529933481	-20.387\\
7.54323725055431	-16.697\\
7.60532150776052	-12.662\\
7.69844789356983	-6.0171\\
7.79157427937915	1.31200000000001\\
7.88470066518846	9.34280000000001\\
7.97782705099777	18.143\\
8.07095343680709	27.791\\
8.1640798226164	37.799\\
8.19512195121951	40.55\\
8.22616407982261	42.123\\
8.25720620842571	41.379\\
8.28824833702882	37.031\\
8.31929046563192	28.714\\
8.38137472283813	8.2594\\
8.41241685144124	1.04669999999999\\
8.44345898004434	-3.55770000000001\\
8.47450110864744	-6.4248\\
8.50554323725055	-8.33449999999999\\
8.53658536585365	-9.7578\\
8.56762749445676	-10.93\\
8.62971175166297	-12.878\\
8.66075388026607	-13.706\\
8.69179600886918	-14.439\\
8.72283813747228	-15.07\\
8.75388026607538	-15.591\\
8.78492239467849	-15.991\\
8.81596452328159	-16.261\\
8.8470066518847	-16.396\\
8.8780487804878	-16.39\\
8.90909090909091	-16.24\\
8.94013303769401	-15.949\\
8.97117516629712	-15.516\\
9.00221729490022	-14.946\\
9.03325942350332	-14.245\\
9.06430155210643	-13.417\\
9.09534368070953	-12.471\\
9.15742793791574	-10.245\\
9.21951219512195	-7.6191\\
9.28159645232816	-4.63740000000001\\
9.34368070953437	-1.3355\\
9.40576496674058	2.25899999999999\\
9.49889135254989	8.15170000000001\\
9.5920177383592	14.554\\
9.68514412416852	20.969\\
9.71618625277162	22.781\\
9.74722838137473	24.121\\
9.77827050997783	24.655\\
9.80931263858093	23.98\\
9.84035476718404	21.771\\
9.87139689578714	18.045\\
9.96452328159646	4.41370000000001\\
9.99556541019956	1.34299999999999\\
10.0266075388027	-0.732609999999994\\
10.0576496674058	-2.05500000000001\\
10.0886917960089	-2.88380000000001\\
10.119733924612	-3.42140000000001\\
10.1507760532151	-3.80189999999999\\
10.2128603104213	-4.381\\
10.3059866962306	-5.20099999999999\\
10.3991130820399	-6.03370000000001\\
10.430155210643	-6.27940000000001\\
10.4611973392461	-6.49029999999999\\
10.4922394678492	-6.6551\\
10.5232815964523	-6.76349999999999\\
10.5543237250554	-6.80690000000001\\
10.5853658536585	-6.7792\\
10.6164079822616	-6.67689999999999\\
10.6474501108648	-6.49930000000001\\
10.6784922394679	-6.24850000000001\\
10.709534368071	-5.929\\
10.7405764966741	-5.5478\\
10.8026607538803	-4.6353\\
10.8957871396896	-3.05160000000001\\
10.9889135254989	-1.46369999999999\\
11.0509977827051	-0.508829999999989\\
11.1130820399113	0.33198999999999\\
11.1751662971175	1.06540000000001\\
11.2372505543237	1.71440000000001\\
11.2993348115299	2.30119999999999\\
11.3614190687362	2.83369999999999\\
11.4235033259424	3.30250000000001\\
11.4545454545455	3.5061\\
11.4855875831486	3.68440000000001\\
11.5166297117517	3.834\\
11.5476718403548	3.95150000000001\\
11.5787139689579	4.03389999999999\\
11.609756097561	4.07910000000001\\
11.6407982261641	4.0855\\
11.6718403547672	4.05289999999999\\
11.7028824833703	3.98169999999999\\
11.7339246119734	3.87360000000001\\
11.7649667405765	3.7312\\
11.7960088691796	3.5581\\
11.8580931263858	3.13820000000001\\
11.9512195121951	2.4049\\
12.0443458980045	1.67609999999999\\
12.1064301552106	1.2501\\
12.1685144124168	0.893540000000002\\
12.1995565410199	0.743380000000002\\
12.230598669623	0.612159999999989\\
12.2616407982262	0.498809999999992\\
12.2926829268293	0.402170000000012\\
12.3237250554324	0.320750000000004\\
12.3547671840355	0.253050000000002\\
12.3858093126386	0.197620000000001\\
12.4168514412417	0.152690000000007\\
12.4478935698448	0.116710000000012\\
12.4789356984479	0.0882310000000075\\
12.509977827051	0.0659609999999873\\
12.5410199556541	0.0488080000000082\\
12.6031042128603	0.0258629999999869\\
12.6651884700665	0.0131020000000035\\
12.7583148558758	0.00435500000000388\\
12.9135254988913	0.000555469999994784\\
13.5033259423503	1.80340009592328e-08\\
14	3.12638803734444e-13\\
};
%\addlegendentry{data1}

\end{axis}
\end{tikzpicture}%
	\end{center}
	\caption{Acoustic pressure along the axis of symmetry at $t=76 \, \mu \textup{s}$ \label{fig:pressure_axis}}
\end{figure}
%%%%%%%%%%%%%%%%%%%%%%%%%%%%%%%%%%%%%%%%%%%%%%%%%%%%%%%%
\section{Conclusion}
In this work, we studied both analytically and numerically the Blackstock equation, which arises in nonlinear acoustics as a model of ultrasound propagation. We showed local-in-time well-posedness of the equation with homogeneous Dirichlet boundary conditions and small initial data. We then proved global well-posedness and exponential decay for the energy of the solution.\\
\indent In addition, we presented a numerical treatment of the equation within the framework of Isogeometric Analysis, where we employed B-Splines as basis functions. Our numerical results confirm the theoretical findings and, in particular, illustrate the influence of the size of the data and the level of dissipation on the behavior of the solution.
%%%%%%%%%%%%%%%%%%%%%%%%%%%%%%%%%%%%%%%%%%%%%%%%%%%%%%%%%%%%%%%%%%%%%%%%%%%%%%%%%%%%%%%%%%%%%%%%%%%%%%%%%%%%%%%%%
\section*{Acknowledgements}
 The authors would like to thank Dr. Pedro Jordan for drawing their attention to the Blackstock equation and Markus Muhr for interesting discussions. We also thank Professor Amiya K. Pani for pointing out a mistake in a previous version of this manuscript. The funds provided by the Deutsche Forschungsgemeinschaft under the grant number WO 671/11-1 are gratefully acknowledged. 
%%%%%%%%%%%%%%%%%%%%%%%%%%%%%%% BIBLIOGRAPHY %%%%%%%%%%%%%%%%%%%%%%%%%%%%%%%%%%%%%%
\bibliography{references}{}
\bibliographystyle{siam} 
\end{document}